\documentclass[12pt,a4paper,reqno]{amsart}

\usepackage{amsfonts,amsthm,amsmath,amssymb,amscd,mathrsfs}
\usepackage{esint}
\usepackage{indentfirst}
\usepackage{cite}
\usepackage{bm,xcolor}
\usepackage{enumitem}
\usepackage{float}
\usepackage{mathtools}
\usepackage[hidelinks]{hyperref} 
\usepackage{graphicx}
\usepackage{accents}
\usepackage{tikz}
\usetikzlibrary{arrows.meta,decorations.markings}
\usepackage{caption}

\definecolor{contour0}{RGB}{68,1,84}
\definecolor{contour1}{RGB}{72,24,106}
\definecolor{contour2}{RGB}{71,45,123}
\definecolor{contour3}{RGB}{66,64,134}
\definecolor{contour4}{RGB}{59,82,139}
\definecolor{contour5}{RGB}{51,99,141}
\definecolor{contour6}{RGB}{44,114,142}
\definecolor{contour7}{RGB}{38,130,142}
\definecolor{contour8}{RGB}{33,145,140}
\definecolor{contour9}{RGB}{31,160,136}
\definecolor{contour10}{RGB}{40,174,128}
\definecolor{contour11}{RGB}{63,188,115}
\definecolor{contour12}{RGB}{94,201,98}
\definecolor{contour13}{RGB}{132,212,75}
\definecolor{contour14}{RGB}{173,220,48}
\definecolor{contour15}{RGB}{216,226,25}
\definecolor{contour16}{RGB}{253,231,37}

\theoremstyle{plain} 
\newtheorem{theorem}{Theorem}[section]
\newtheorem{remark}{Remark}[section]
\newtheorem{lemma}[theorem]{Lemma}
\newtheorem{prop}[theorem]{Proposition}
\newtheorem{coro}[theorem]{Corollary}

\theoremstyle{definition}
\newtheorem{definition}{Definition}[section]
\newtheorem{example}[definition]{Example}

\def\mbN{\mathbb{N}}
\def\mbR{\mathbb{R}}
\def\mbS{\mathbb{S}}
\def\mbT{\mathbb{T}}
\def\mbZ{\mathbb{Z}}

\def\be{\mathbf{e}}
\def\bg{\mathbf{g}}

\def\bv{\mathbf{v}}
\def\bw{\mathbf{w}}

\def\bF{\mathbf{F}}

\def\mcD{\mathcal{D}}
\def\mcH{\mathcal{H}}

\def\mcL{\mathcal{L}}
\def\mcN{\mathcal{N}}
\def\mcR{\mathcal{R}}

\def\mcT{\mathcal{T}}

\def\msN{\mathscr{N}}
\def\msR{\mathscr{R}}

\def\Div{{\rm div}}

\DeclareMathOperator{\sgn}{sgn}

\DeclareMathOperator{\dist}{dist}

\newcommand{\Rn}[1]{\expandafter{\romannumeral #1\relax}}
\newcommand{\osc}{\operatorname{osc}}
\newcommand{\BMO}{\mathrm{BMO}}
\newcommand{\btau}{\boldsymbol{\tau}}
\newcommand{\bga}{\boldsymbol{\gamma}}

\begin{document}

\title[Pressure estimates and classifications for steady Euler flows]{Pressure estimates and classifications for steady $H^1_{\rm loc}$-solutions to two-dimensional incompressible Euler equations}

\author{Changfeng Gui}
\address{Department of Mathematics, Faculty of Science, University of Macau, Taipa, Macao.}
\email{changfenggui@um.edu.mo}

\author{Qinfeng Li}
\address{School of Mathematics, Hunan University, Changsha, P.R. China.}
\email{liqinfeng1989@gmail.com}

\author{Chunjing Xie}
\address{School of Mathematical Sciences,  Ministry of Education Key Laboratory of Scientific and Engineering Computing, CMA-Shanghai, Shanghai Jiao Tong University, Shanghai 200240, China.}
\email{cjxie@sjtu.edu.cn}

\author{Huan Xu}
\address{Department of Mathematics, Faculty of Science, University of Macau, Taipa, Macao.}
\email{hzx0016@auburn.edu}

\date{}

\begin{abstract}
We study two-dimensional steady incompressible Euler flows with finite total curvature. A distinguished family of such flows arises from finite-Morse-index solutions to semilinear elliptic equations in two dimensions. We establish three main properties of finite-total-curvature flows in the $L^\infty\cap H^1_{\rm loc}$ class. First, in the plane and the half-plane, every such flow without stagnation points is a parallel shear flow. Second, in the plane, the half-plane, and the infinite strip, the pressure converges uniformly to a constant at each end at infinity. Third, in the plane, the half-plane, the infinite strip, and the periodic strip, we establish sharp lower bounds for the total curvature in terms of the pressure oscillation.

The rigidity result in fact extends to a broader class of flows: in the plane and the half-plane, any steady Euler flow without stagnation points is a parallel shear flow provided that $\nabla P\in L^q$ for some $1\le q\le 2$. Combined with standard elliptic estimates for the pressure, this result yields a concise new proof of Hamel and Nadirashvili's rigidity theorems \cite{HN_CPAM_2017,HN_ARMA_2019}, with the regularity assumption sharpened optimally to $H^1_{\rm loc}$.

The pressure serves as a unifying quantity throughout the analysis.
\end{abstract}



\maketitle



\section{Introduction}

This paper studies the two-dimensional steady incompressible Euler equations
\begin{eqnarray}\label{Euler}
\left\{\begin{aligned}
&\bv\cdot\nabla\bv+\nabla P=0, & \hbox{in } \Omega, \\
&\Div\,\bv=0, & \hbox{in } \Omega,
\end{aligned}\right.
\end{eqnarray}
where $\bv=(v_1,v_2)$ is the velocity field and $P$ is the pressure.
The domain $\Omega\subset\mbR^2$ is one of the following basic domains:
\[
\mbR^2,\quad
\mbR^2_+:=\mbR\times(0,\infty),\quad
\Omega_\infty:=\mbR\times(-1,1),\quad
\Omega_{\rm per}:=\mbT\times(-1,1),
\]
where $\mbT:=\mbR/2\pi\mbZ$.
On domains with horizontal flat boundaries, we impose the {\it slip boundary condition}
\begin{equation}\label{SBC}
v_2=0 \quad \hbox{on } \partial\Omega.
\end{equation}
We impose only local $H^1$-regularity on the solution throughout, so the above boundary condition is understood in the sense of trace.

The vorticity form of \eqref{Euler}, written solely in terms of $\bv$, reads
\begin{equation}\label{divergence-form Euler}
\Div(v_1\nabla v_2-v_2\nabla v_1)=0.    
\end{equation}
Using this equation, the authors of \cite{GXX_CMP_2026} classified steady Euler flows in $\mbR^2$, $\mbR^2_+$, and $\Omega_\infty$ in terms of the velocity directions by estimating the {\it total curvature}
\begin{align*}
\mcT(\bv,\Omega):=\int_{\Omega}\frac{|v_1\nabla v_2-v_2\nabla v_1|^2}{|\bv|^2}\,dx,
\end{align*}
where the integrand is defined to be zero on the stagnation set $\{|\bv|=0\}$.
Total curvature arises naturally from energy estimates for \eqref{divergence-form Euler}, and it is particularly well-suited for studying rigidity, as a flow is parallel shear precisely when its total curvature vanishes.
Here, $\bv$ is called a parallel shear flow if there exist a unit vector $\be$ and a scalar function $V$ such that
\[
\bv(x)=V(x\cdot \be^\perp)\be.
\]
In the half-plane or the strip with the slip boundary condition \eqref{SBC}, a parallel shear flow necessarily has the form
\[
\bv(x_1,x_2)=(V(x_2),0).
\]

This work studies the class of flows with finite total curvature.
This class contains important flows generated by semilinear elliptic equations. 
Indeed, if $u$ solves a semilinear elliptic equation
\begin{equation}\label{slee}
\Delta u=F'(u) \quad \hbox{in } \mbR^2,    
\end{equation}
and
\[
\bv=\nabla^\perp u, \quad P=F(u)-\frac12|\nabla u|^2,
\]
then $(\bv,P)$ is a solution to \eqref{Euler}.
It is well-known that finite-Morse-index solutions to \eqref{slee} generate finite-total-curvature flows (see, e.g., \cite{WW_CPAM_2019} and the references therein).
In particular, finite-Morse-index solutions of the Allen--Cahn equation \[
\Delta u=u^3-u
\]
have been studied extensively in \cite{DKP_TAMS_2013,DKPW_JFA_2010,Gui_JDE_2012,KLP_NHM_2012,KLPW_CVPDE_2015,Wang_NoDEA_2017,WW_CPAM_2019}, providing a distinguished family of examples of finite-total-curvature flows.

We are interested in qualitative and quantitative properties of finite-total-curvature flows, with a particular focus on rigidity, far-field asymptotics, and sharp lower bounds on the total curvature.
Throughout our analysis, pressure serves as the central quantity.

Regarding rigidity, most existing results are naturally formulated in terms of the velocity
\cite{DER_CMP_2026,EFR_arXiv_2025,HN_CPAM_2017,HN_ARMA_2019,HN_JEMS_2023,LLSX_Poincare_2026,Ruiz_ARMA_2023,WZ_arXiv_2023}, the stream function \cite{EHSX_Duke_2026}, the vorticity \cite{CEW_ARMA_2023,GPSY_Duke_2021,LZ_ARMA_2011}, or the geometry of streamlines \cite{DN_CPAM_2026,GXX_arXiv_2026}. A representative result of Hamel and Nadirashvili states that a steady Euler flow in the plane, the half-plane, or the strip, with a positive lower bound on its speed must be a parallel shear flow \cite{HN_CPAM_2017,HN_ARMA_2019}.
The strictly positive lower bound assumption cannot be removed \cite{DR_CVPDE_2025}.

We show that any bounded finite-total-curvature flow in the plane or the half-plane, without essential stagnation points, must be a parallel shear flow. This rigidity result in fact extends to a broader class of flows. 
Indeed, by the Euler system \eqref{Euler}, the total curvature admits the reformulation
\begin{align*}
\mcT(\bv,\Omega)=\int_{\Omega}\frac{|\nabla P|^2}{|\bv|^2}\,dx.
\end{align*}
Consequently, any bounded flow with finite total curvature satisfies $\nabla P\in L^2$.
We can establish a stronger rigidity theorem that any steady Euler flow in the plane or the half-plane without essential stagnation points must be a parallel shear flow, provided that $\nabla P\in L^q$ for some $q\in[1,2]$.
This result requires neither boundedness of $|\bv|$ nor a uniform positive lower bound on $|\bv|$.
However, when the two-sided positive bounds on $|\bv|$ are imposed, standard elliptic estimates yield $\nabla P\in L^2$.  
Thus, we obtain a concise new proof of Hamel and Nadirashvili's rigidity results for merely $H^1_{\rm loc}$ solutions. Remarkably, this local $H^1$ regularity requirement is sharp: the Euler flow $\bv=\frac{x^\perp}{|x|}$ belongs to $W^{1,q}_{\rm loc}$ for every $q\in[1,2)$, has constant magnitude $1$, yet is not a parallel shear flow.

We turn next to the far-field asymptotic behavior of finite-total-curvature flows in $\mbR^2$, $\mbR^2_+$, and $\Omega_\infty$. 
Due to the presence of stagnation points, obtaining the asymptotic behavior of the velocity field seems formidable. 
We instead focus on the pressure and show that it converges uniformly to a constant at each end at infinity.

Lastly, we study lower bound estimates for the total curvature.
This problem was initiated in \cite{GXX_CMP_2026} for steady flows in the half-plane and the infinite strip, where bounded Euler flows are classified into shears and flows whose velocity directions form the full circle or a closed semicircle. 
The classification relies on a sharp lower bound of the total curvature involving a boundary quantity, with equality attained precisely by shears and flows whose set of directions is a closed semicircle.
The latter thus forms the least nontrivial class, further studied in \cite{GRXX_ARMA_2026}.
However, since this lower bound depends on the far‑field asymptotics of the horizontal velocity trace, it admits no natural analogue in domains like the whole plane or the periodic strip. 
Our main quantitative advance is to replace that boundary‑dependent term with the pressure oscillation
\[
\osc_\Omega P:=\sup_\Omega P-\inf_\Omega P.
\]
This substitution is natural, since zero total curvature and constant pressure each characterize a parallel shear flow.
Thus, both the total curvature and pressure oscillation measure the deviation of a steady Euler flow from the shear regime.
The resulting estimates are more intrinsic to the Euler flow and apply to all geometries considered here.

\subsection{Main results}
The first main result, stated below, concerns the qualitative and quantitative behavior of finite-total-curvature flows.

\begin{theorem}\label{thm_FTC flow}
Let $\Omega\in \{ \mbR^2, \mbR_+^2, \Omega_\infty, \Omega_{\rm per} \}$.
Let $\bv\in H^1_{\textup{loc}}(\Bar{\Omega})\cap L^\infty(\Omega)$ and $P\in H^1_{\textup{loc}}(\Bar{\Omega})$ solve \eqref{Euler}.
Except for $\Omega=\mbR^2$, the flow is supplemented with the slip boundary condition \eqref{SBC} in the sense of trace.
If the total curvature is finite, then the following statements hold.
\begin{enumerate}[label=\textup{(\roman*)}]
\item For $\Omega\in \{ \mbR^2, \mbR_+^2\}$, assume that $\bv$ has no essential stagnation points in $\Bar{\Omega}$ in the following sense: for any $K\Subset\Bar{\Omega}$, there exists $c_K>0$ such that $|\bv|\ge c_K$ a.e. in $K$. Then $\bv$ is a shear flow.

\item The pressure $P$ belongs to $C(\Bar{\Omega})\cap L^\infty(\Omega)$. Moreover, for $\Omega\in \{ \mbR^2, \mbR_+^2\}$, there exists a constant $P_\infty$ such that
\[
\lim_{|x|\to\infty}P(x)=P_\infty.
\]
For $\Omega=\Omega_\infty$, there exist constants $P_+$ and $P_-$ such that
\[
\lim_{x_1\to\pm\infty}P(x)=P_\pm \quad \hbox{ uniformly in  } x_2\in[-1,1].
\]

\item It holds that
\begin{align}\label{ineq_sharp total curvature}
\mcT(\bv,\Omega) \ge c_{\Omega} \cdot \osc_{\Omega}P,
\end{align}
where
\[
c_\Omega=
\begin{cases}
2\pi, &\text{if} \ \Omega=\mbR^2 \ or \ \Omega_{\rm per},\\
\pi, &\text{if} \ \Omega=\mbR^2_+ \ or \ \Omega_\infty.
\end{cases}
\]
Moreover, there exists a nonshear steady Euler flow, satisfying all the assumptions of the theorem, for which equality in \eqref{ineq_sharp total curvature} is attained.
\end{enumerate}
\end{theorem}

Since a bounded finite-total-curvature flow satisfies $\nabla P\in L^2$, the rigidity result in Theorem \ref{thm_FTC flow} will follow from the following stronger rigidity theorem.

\begin{theorem}\label{thm_new rigidity}
Suppose $\Omega=\mbR^2$ or $\mbR^2_+$. Let $\bv\in H^1_{\rm loc}(\Bar{\Omega})$ and $P\in H^1_{\rm loc}(\Bar{\Omega})$ solve \eqref{Euler}; in the half-plane, impose the slip boundary condition \eqref{SBC}. 
Suppose $\bv$ has no essential stagnation points in $\Bar{\Omega}$.
If $\nabla P\in L^q(\Omega)$ for some $q\in[1,2]$, then $P$ is constant; equivalently, $\bv$ is a parallel shear flow.
\end{theorem}

A few remarks are in order.

\begin{remark}
Theorem \ref{thm_new rigidity} requires neither $\bv\in L^\infty$ nor $\inf_\Omega|\bv|>0$.
\end{remark}

\begin{remark}
In an infinite strip, this rigidity result no longer holds. Indeed, De Regibus and Ruiz \cite{DR_CVPDE_2025} constructed bounded non-shear flows without stagnation points from monotone heteroclinic solutions of semilinear elliptic equations; these flows have finite total curvature \cite{GXX_CMP_2026}, and hence satisfy $\nabla P\in L^2$.
\end{remark}

\begin{remark}
It remains unclear whether the integrability exponent $q=2$ is optimal. See Remark \ref{remark_two is sharp} for further discussion.
\end{remark}

As a corollary of Theorem \ref{thm_new rigidity}, we recover Hamel--Nadirashvili's rigidity theorems \cite{HN_CPAM_2017,HN_ARMA_2019} with a new proof, and relax the regularity assumption from $C^2$ to the optimal $H^1$.

\begin{coro}\label{coro_HN rigidity}
Let $\Omega\in\{\mbR^2,\mbR^2_+,\Omega_{\infty}\}$.  
Suppose $\bv\in H^1_{\rm loc}(\Bar{\Omega})$ and $P\in H^1_{\rm loc}(\Bar{\Omega})$ solve \eqref{Euler}, together with the slip boundary condition \eqref{SBC} when a boundary is present. Assume that for some constants $0<m\le M<\infty$,
\begin{equation*}
 m\le |\bv(x)|\le M \quad \text{for a.e. } x\in\Omega.
\end{equation*}
Then $\bv$ is a parallel shear flow.
\end{coro}

\begin{remark}
The Sobolev $H^1$-regularity assumption cannot be lowered.  
Indeed, define
\begin{equation*}
 \bv(x)=\frac{x^\perp}{|x|}, \qquad P(x)=\log|x|.
\end{equation*}
Then $(\bv,P)$ is a non-shear solution of \eqref{Euler} in $\mbR^2$ with $|\bv|\equiv1$. Moreover,
\[
 \bv,P\in W^{1,q}_{\rm loc}(\mbR^2)\qquad\text{for every } 1\le q<2.
\]
The failure of rigidity is due to a point-degree defect in the pressure equation
\begin{equation*}
 \Div\left(\frac{\nabla P}{|\bv|^2}\right)
 =\Delta\log|x|=2\pi\delta_0.
\end{equation*}
\end{remark}

\begin{remark}
In the strip, the upper bound of $|\bv|$ is not needed for $C^2$ solutions (see \cite{HN_CPAM_2017}).
\end{remark}


\subsection{Key ingredients of the proofs}
Pressure is the central quantity throughout our analysis, which sets the present work apart from the existing literature.
Let us now describe the key ingredients underlying our main proofs.

\subsubsection{Rigidity via maximum principle and elliptic regularity}
An advantage of proving rigidity for $H^1$-solutions is that rigidity in the whole plane implies rigidity in other domains, since an elementary reflection sends the original solution to $H^1$-solutions in the whole plane.

In the whole plane, if $\bv$ has no essential stagnation points, $P$ is a weak solution to
\begin{equation}\label{eq_elliptic P}
\Div\left(\frac{\nabla P}{|\bv|^2}\right)=0 \quad \text{in } \mbR^2.
\end{equation}
The assumption that $\nabla P\in L^q(\mbR^2)$ for some $q\in[1,2]$ implies that there are good radii $r_j\to\infty$ such that
\[
\osc_{\partial B_{r_j}}P\to 0.
\]
Then Theorem \ref{thm_new rigidity} follows from the maximum principle for weak solutions.

The first equation of the Euler system \eqref{Euler} can be reformulated as
\begin{align*}
-\nabla P=\Div(\bv\otimes\bv).
\end{align*}
If $\bv$ is bounded, the endpoint Calder\'on--Zygmund theory yields $P \in \BMO(\mbR^2)$. If, in addition, $|\bv|$ has a positive lower bound, applying Caccioppoli estimates to \eqref{eq_elliptic P}, together with the John--Nirenberg inequality, lifts this $\BMO$ regularity to $\nabla P\in L^2(\mbR^2)$. Consequently, Corollary \ref{coro_HN rigidity} follows from Theorem \ref{thm_new rigidity}.

\subsubsection{Pressure oscillation estimates}
For clarity, we illustrate the key idea in the smooth setting. The $H^1$ case is technically more delicate and requires the Sobolev coarea formula \cite{MSZ_TAMS_2003}, the level-set structure of $W^{2,1}$ functions and their approximation by $C^1$ functions \cite{BKK_RMI_2013}.

First, away from stagnation points, the Euler system implies that
\begin{equation*}
\frac{|v_1\nabla v_2-v_2\nabla v_1|^2}{|\bv|^2}
=|\nabla P|\,\left|\nabla\left(\frac{\bv}{|\bv|} \right)\right|.
\end{equation*}
Then the coarea formula gives
\begin{equation}\label{coarea1}
\mcT(\bv,\Omega)=
\int_{\mbR}
\left(
\int_{P^{-1}(t)\cap\{ |\bv|>0 \}}\left|\nabla\left(\frac{\bv}{|\bv|} \right)\right|\,ds
\right)dt.
\end{equation}
Locally, $\bv$ admits a phase function $\theta$. Direct computation yields
\begin{align*}
\left|\nabla\left(\frac{\bv}{|\bv|} \right)\right|
=|\nabla\theta|.
\end{align*}
Moreover, the Euler equations imply that $\nabla P$ and $\nabla\theta$ are orthogonal. Consequently, the coarea integral (i.e., the inner integral on the right-hand side of \eqref{coarea1}) represents the phase increment of the flow along $P^{-1}(t)$.
We shall show that this phase increment has a positive lower bound, uniformly for a.e. $t$; consequently, the pressure oscillation is controlled by the total curvature.

\begin{remark}
There is also a dual direction-level representation for the total curvature. Since $\nabla P\perp\nabla\theta$, the coarea formula applied to the direction map $\bv/|\bv|$ gives
\begin{equation}\label{coarea2}
\mcT(\bv,\Omega)=
\int_{\mbS^1}
\int_{\Sigma_y}|\nabla P|\,ds\,d\theta(y),
\end{equation}
where
\begin{align*}
\Sigma_y:=\left\{x\in\Omega: |\bv(x)|>0, \frac{\bv(x)}{|\bv(x)|}=y\right\}.    
\end{align*}
This formula is closely related to the equidistribution property of the total curvature established in \cite{GXX_CMP_2026}.    
\end{remark}

This paper is organized as follows. In Section \ref{sec_prelim}, we collect some preliminaries. Section \ref{sec_rigidity} proves the rigidity results in Theorem \ref{thm_new rigidity} and Corollary \ref{coro_HN rigidity}. Section \ref{sec_pressure asymptotics} establishes the far-field asymptotics of the pressure for finite total curvature flows in the plane, the half-plane, and the strip. Section \ref{sec_lower bound} proves the total curvature estimates in terms of the pressure oscillation, whose sharpness is demonstrated in Section \ref{sec_equality case} by concrete examples.

\section{Preliminaries}\label{sec_prelim}
First, we show that an elementary reflection maps $H^1_{\rm loc}$-solutions in the half-plane or strip to solutions in the whole plane with the same regularity.
\begin{lemma}\label{lemma_reflection}
Let $\Omega=\mbR^2_+$ or $\Omega_\infty$. Suppose $\bv\in H^1_{\rm loc}(\Bar{\Omega})$ and $P\in H^1_{\rm loc}(\Bar{\Omega})$ solve \eqref{Euler}-\eqref{SBC}. 
\begin{enumerate}[label=\textup{(\roman*)}]
\item For $\Omega=\mbR^2_+$, define $\widetilde\bv=(\widetilde v_1,\widetilde v_2)$ and $\widetilde P$ by
\begin{equation*}
 \widetilde v_1(x)=v_1(x_1,|x_2|), \ 
 \widetilde v_2(x)=\sgn(x_2) v_2(x_1,|x_2|), \ 
 \widetilde P(x)=P(x_1,|x_2|).
\end{equation*}
Then $\widetilde\bv\in H^1_{\rm loc}(\mbR^2)$ and $\widetilde P\in H^1_{\rm loc}(\mbR^2)$ solve \eqref{Euler} in $\mbR^2$.

\item For $\Omega=\Omega_\infty$, successive reflections across the lines $x_2=2k+1$, $k\in\mathbb Z$, extending $v_1$ and $P$ evenly and $v_2$ oddly across each reflection line, produce a pair $(\widetilde\bv,\widetilde P)$ on $\mbR^2$ that is $4$-periodic in $x_2$, such that $\widetilde\bv\in H^1_{\rm loc}(\mbR^2)$ and $ \widetilde P\in H^1_{\rm loc}(\mbR^2)$ solve \eqref{Euler} in $\mbR^2$. 
\end{enumerate}
\end{lemma}
\begin{proof}
The even extension of an $H^1$ function preserves the $H^1$ regularity. The odd extension has the same property when the trace of the function on the flat boundary vanishes. Hence the slip boundary condition implies $ \widetilde\bv\in H^1_{\rm loc}(\mbR^2)$ and $ \widetilde P\in H^1_{\rm loc}(\mbR^2)$. It is straightforward to verify that $( \widetilde\bv, \widetilde P)$ solves \eqref{Euler} in $\mbR^2$.
\end{proof}

Next, we show that the pressure has more regularity than assumed a priori, owing to a Jacobian structure in the pressure equation. The following lemma is known (see, e.g., \cite{KPR_ARMA_2013} for a closely related result).

\begin{lemma}\label{lemma_pressure in W21}
Let $\bv\in H^1_{\rm loc}(\mbR^2)$ and $P\in H^1_{\rm loc}(\mbR^2)$ be a solution to \eqref{Euler}. Then $P\in W^{2,1}_{\rm loc}(\mbR^2)$. In particular, $P$ admits a continuous representative.
\end{lemma}
\begin{proof}
By the Euler system \eqref{Euler}, one has
\[
\Delta P=-\sum_{i,j=1}^2\partial_i v_j\,\partial_j v_i=2\det\nabla\bv.
\]
For any $U\Subset\mathbb R^2$, choose $\eta\in C_c^\infty(\mbR^2)$ such that $\eta=1$ on a
neighborhood of $\Bar{U}$, and let $\bw=\eta\bv$. Then $\bw\in H^1(\mbR^2)$ and
\[
F:=2\det\nabla\bw=2\det\nabla\bv \qquad \text{a.e. on } U.
\]
The Coifman--Lions--Meyer--Semmes Jacobian estimate \cite{CLMS_JMPA_1993} gives that $F$ belongs to the real Hardy space $H^1_{\rm Har}(\mbR^2)$.

Let $\Gamma(x)=(2\pi)^{-1}\log|x|$ and put $Q=\Gamma*F$. Then $\Delta Q=F$, while each component of $\nabla^2Q$ is a double Riesz transform of $F$. The endpoint Calder\'on--Zygmund estimate (see, e.g., \cite{Grafakos_modern_2009}) therefore yields $\nabla^2 Q\in L^1(\mbR^2)$. Since $\Gamma$ and $\nabla\Gamma$ are locally integrable and $F$ is compactly supported, it follows that $Q\in W^{2,1}_{\mathrm{loc}}(\mathbb R^2)$.

On $U$ we have $\Delta(P-Q)=0$. Hence $P-Q$ is smooth on $U$, and consequently $P\in W^{2,1}(U)$. Since $U\Subset\mathbb R^2$ is arbitrary, the conclusion follows.
\end{proof}

In the next two lemmas, we gather the consequences of \cite{BKK_RMI_2013} concerning the level-set structure of $W^{2,1}$ functions and their approximation by $C^1$ functions.

\begin{lemma}[\hspace{0.05em}\cite{BKK_RMI_2013}]\label{lemma_BKK_Sard}
Let $\Omega\subset\mbR^2$ be a bounded Lipschitz domain. If $P\in W^{2,1}(\Omega)$, there is an $\mcL^1$-null set $\msN_0\subset\mbR$ such that, for every $t\in\mbR\setminus\msN_0$, the level set $P^{-1}(t)$ is a finite disjoint union of $C^1$ loops or simple arcs ending transversally on the boundary, and the tangent vector on each component is absolutely continuous.    
\end{lemma}

\begin{lemma}[\hspace{0.05em}\cite{BKK_RMI_2013}]\label{lemma_C1_approx}
Let $\Omega\subset\mbR^2$ be a bounded Lipschitz domain. If $P\in W^{2,1}(\Omega)$, then for every $\varepsilon>0$ there are an open set $V_\varepsilon\subset\mbR$ with $\mcH^1(V_\varepsilon)<\varepsilon$ and a function $P_\varepsilon\in C^1(\mbR^2)$ such that, on $P^{-1}(\mbR\setminus V_\varepsilon)$,
\begin{equation*}
P=P_\varepsilon, \quad \nabla P=\nabla P_\varepsilon\neq0.
\end{equation*}    
\end{lemma}

Throughout, every locally integrable function is identified with its precise representative, i.e., the canonical representative determined by local averages.

\begin{lemma}\label{lemma_trace on good level}
Let $\Omega\subset\mbR^2$ be a bounded Lipschitz domain. Suppose that $\bv\in H^1(\Omega)$ and $P\in W^{2,1}(\Omega)$. Let $\msN_0$ be the $\mcL^1$-null set in Lemma \ref{lemma_BKK_Sard}. Then there is an $\mcH^1$-null set $\msN\subset\mbR$ containing $\msN_0$ such that, for every $t\in\mbR\setminus\msN$ and every connected component $\Gamma$ of $P^{-1}(t)$, the following assertions hold.
\begin{enumerate}[label=\textup{(\roman*)}]
\item It holds that
\[
\bv_\Gamma:=\bv|_\Gamma\in H^1(\Gamma), \qquad
(\nabla \bv)|_\Gamma \in L^2(\Gamma),
\]
where the restriction $\bv_\Gamma$ agrees with the ordinary Sobolev trace of $\bv$ for $\mcH^1$-almost every point on $\Gamma$. Denote by $\btau$ the unit tangent field on $\Gamma$, and by $\partial_s$ the tangential derivative along the direction of $\btau$. Then
\[
\partial_s \bv_\Gamma=(\nabla \bv)|_\Gamma\,\btau, \quad \text{$\mcH^1$-a.e. on $\Gamma$}.
\]

\item 
If $\Gamma$ is an arc and $x\in\partial\Gamma\subset\partial\Omega$, then
\begin{equation*}
\operatorname{Tr}_{\partial\Gamma}\bv_\Gamma(x)=\bigl(\operatorname{Tr}_{\partial\Omega}\bv\bigr)(x),
\end{equation*}
where $\operatorname{Tr}$ denotes the trace operator.

\item After enlarging $\msN$ by another $\mcH^1$-null set, the above results are independent of the choices of representatives of $\bv$, $\nabla\bv$ and $\operatorname{Tr}_{\partial\Omega}\bv$.
\end{enumerate}
\end{lemma}

\begin{proof}
By Lemma \ref{lemma_BKK_Sard}, for every $t\in\mbR\setminus\mcN_0$, the level set $P^{-1}(t)$ is a finite disjoint union of $C^1$ loops or simple arcs ending transversally on the boundary.

By Lemma \ref{lemma_C1_approx}, for every $k\in\mbN_+$, there exist an open set $V_k\subset\mbR$ with $\mcH^1(V_k)<\frac1k$ and a function $P_k\in C^1(\mbR^2)$ such that
\begin{equation*}
P=P_k, \quad \nabla P=\nabla P_k\neq0, \quad \text{on} \ P^{-1}(\mbR\setminus V_k).
\end{equation*}    
Let
\[
\mcN_1=\bigcap_{k=1}^{\infty} V_k.
\]
Then $\mcH^1(\mcN_1)=0$. For every $t\in\mbR\setminus\mcN_1$, there exists $k\in\mbN_+$ such that $t\in \mbR\setminus V_k$.

(\Rn{1}) 
By density of smooth functions in $H^1(\Omega)$, there are $\bv_j\in C^\infty(\overline\Omega)$ such that
\begin{equation}\label{eq:summable-approximation}
\sum_{j=1}^{\infty}\|\bv_j-\bv\|_{H^1(\Omega)}^2<\infty.
\end{equation}
Define the nonnegative measurable function
\[
F_j(x):=|\bv_j(x)-\bv(x)|^2+|\nabla \bv_j(x)-\nabla \bv(x)|^2.
\]
The Sobolev coarea formula \cite{MSZ_TAMS_2003,PR_Selecta_2020} gives
\begin{equation}\label{eq:weighted-coarea}
\int_{\{|\nabla P|>0\}}F_j(x)\,dx
=\int_{\mbR}\int_{P^{-1}(t)\cap\{|\nabla P|>0\}}\frac{F_j(x)}{|\nabla P(x)|}\,d\mcH^1\,dt.
\end{equation}
Let
\[
Q_j(t):=\int_{P^{-1}(t)\cap\{|\nabla P|>0\}}\frac{F_j(x)}{|\nabla P(x)|}\,d\mcH^1.
\]
By \eqref{eq:summable-approximation} and \eqref{eq:weighted-coarea},
\begin{align*}\label{eq:sum-Q}
\int_\mbR\sum_{j=1}^{\infty}Q_j(t)\,dt
= \sum_{j=1}^{\infty}\int_{\{|\nabla P|>0\}}F_j(x)\,dx
\le \sum_{j=1}^{\infty}\|\bv_j-\bv\|_{H^1(\Omega)}^2<\infty.
\end{align*}
It follows that there is a null set $\msN_2\subset\mbR$ such that
\begin{equation}\label{eq:Q-convergence}
\sum_{j=1}^{\infty}Q_j(t)<\infty, \quad \text{and hence} \quad Q_j(t)\longrightarrow0,
\end{equation}
for every $t\in\mbR\setminus \msN_2$.

Fix
\[
t\in\mbR\setminus(\msN_0\cup \msN_1 \cup \msN_2)
\]
and let $\Gamma$ be any component of $P^{-1}(t)\cap\Bar{\Omega}$. By Lemma~\ref{lemma_C1_approx}, $|\nabla P|$ is continuous and positive on $\Gamma$. Therefore
\begin{align*}
\int_\Gamma \left(|\bv_j-\bv|^2+|\nabla \bv_j-\nabla \bv|^2\right)\,ds
\le \|\nabla P\|_{L^\infty(\Gamma)} Q_j(t) \rightarrow 0.
\end{align*}
This implies that
\begin{equation}\label{eq:two-convergences}
\bv_j|_\Gamma \longrightarrow \bv|_\Gamma \quad \text{in }L^2(\Gamma), \quad
(\nabla \bv_j)|_\Gamma \longrightarrow (\nabla \bv)|_\Gamma \quad \text{in }L^2(\Gamma).
\end{equation}
In particular, both $\bv|_\Gamma$ and $(\nabla \bv)|_\Gamma$ belong to $L^2(\Gamma)$.
By the continuity of the trace operator
\[
\operatorname{Tr}_\Gamma: H^1(\Omega) \longrightarrow L^2(\Gamma)
\]
and the smoothness of $\bv_j$, one has
\[
\bv_j|_\Gamma=\operatorname{Tr}_\Gamma \bv_j \longrightarrow \operatorname{Tr}_\Gamma \bv \quad\text{in }L^2(\Gamma).
\]
This gives that $\bv|_\Gamma=\operatorname{Tr}_\Gamma \bv$ in $L^2(\Gamma)$. 

Let $\bga: I\longrightarrow \Gamma$ be an arclength parametrization. For an arc take $I=(0,\ell)$, where $\ell=\mcH^1(\Gamma)$; for a loop take the one-dimensional torus $I=\mbR/\ell\mbZ$. Let $\btau(s)=\bga'(s)$.
The classical chain rule gives
\[
\partial_s (\bv_j|_\Gamma)=(\nabla \bv_j)|_\Gamma\,\btau.
\]
It follows from \eqref{eq:two-convergences} that
\[
\partial_s (\bv_j|_\Gamma) \longrightarrow (\nabla \bv)|_\Gamma\,\btau \quad\text{in } L^2(\Gamma).
\]
This, together with \eqref{eq:two-convergences}, further implies that
\[
\bv|_\Gamma \in H^1(\Gamma), \quad 
\partial_s(\bv|_\Gamma)=(\nabla \bv)|_\Gamma\,\btau \quad \mcH^1 \text{-almost everywhere on } \Gamma.
\]
So we have proved part (\Rn{1}).

(\Rn{2}) 
For simplicity of notation, let us introduce
\[
p:=P|_{\partial\Omega}, \quad \bg:=\operatorname{Tr}_{\partial\Omega}\bv, \quad \bg_j:=\bv_j|_{\partial\Omega}.
\]
The trace theorem gives $p\in W^{1,1}(\partial\Omega)$ (see, e.g., \cite{KJF_book_1977}).
Moreover, continuity of the trace map $H^1(\Omega)\to L^2(\partial\Omega)$ and \eqref{eq:summable-approximation} imply
\begin{equation}\label{eq:summable-boundary-traces}
\sum_{j=1}^{\infty}\|\bg_j-\bg\|_{L^2(\partial\Omega)}^2<\infty.
\end{equation}
Choose the precise representative of $\bg$, and denote the arclength derivative of $p$ by $p'$.  For $t\in\mbR$, define
\[
R_j(t):=\sum_{x\in p^{-1}(t): \, |p'(x)|>0} \frac{|\bg_j(x)-\bg(x)|^2}{|p'(x)|}.
\]
The area formula and \eqref{eq:summable-boundary-traces} give
\begin{align*}
\int_\mbR\sum_{j=1}^{\infty}R_j(t)\,dt
=\sum_{j=1}^{\infty} \int_{\partial\Omega\cap\{|p'|>0\}}|\bg_j-\bg|^2\,d\mcH^1<\infty.
\end{align*}
Consequently, there is an $\mcH^1$-null set $\msN_3\subset\mbR$ such that
\begin{equation}\label{eq:boundary-pointwise-convergence}
R_j(t)\longrightarrow 0 \quad \hbox{ for every } t\in\mbR\setminus\msN_3.
\end{equation}
It follows that, whenever $t\notin \msN_3$,
\begin{equation}\label{eq:boundary-values-converge}
\bg_j(x)\longrightarrow \bg(x) \qquad \text{for every } x \in p^{-1}(t).
\end{equation}

Suppose finally that $\Gamma$ is an arc. The proof of part (\Rn{1}) gives that $\bv_j|_\Gamma$ converges to $\bv|_\Gamma$ in $H^1(\Gamma)$, hence uniformly converges by Sobolev embedding.
Hence, one has
\[
\operatorname{Tr}_{\partial L}\bv_L(x) 
=\lim_{j\to\infty}\bv_j(x)
=\lim_{j\to\infty}\bg_j(x)
=\bg(x)=\bigl(\operatorname{Tr}_{\partial\Omega}\bv\bigr)(x).
\]
So part (\Rn{2}) is also proved.

(\Rn{3})
Suppose that one changes a representative of $\bv$ or $\nabla \bv$ on an $\mcL^2$-null set $E\subset\Omega$. The coarea formula implies that for almost every $t$,
\[
\int_{E\cap P^{-1}(t)\cap\{|\nabla P|>0\}}\frac{1}{|\nabla P|}\,d\mcH^1=0.
\]
On every good level $|\nabla P|$ is continuous and strictly positive, so $\mcH^1(E\cap P^{-1}(t))=0$. Hence the coarea restrictions are independent of all choices of planar representatives, after enlarging $\msN$ by another null set.  

Suppose that $\bg_1$ and $\bg_2$ are two representatives of the trace of $\bv$ such that they only differ on an $\mcH^1$-null set $G\subset\partial\Omega$. The area formula implies that for almost every $t$,
\[
\sum_{x\in G\cap p^{-1}(t)\cap \{ |p'|>0\} } \frac{1}{|p'(x)|}=0.
\]
This means $G\cap p^{-1}(t)\cap \{ |p'|>0\}$ is an empty set.
Consequently, $\bg_1=\bg_2$ on $p^{-1}(t)\cap \{ |p'|>0\}$.
This finishes the proof.  
\end{proof}

\section{Rigidity for non-stagnating flows}\label{sec_rigidity}

This section is devoted to the proof of Theorem \ref{thm_new rigidity} and Corollary \ref{coro_HN rigidity}.
Thanks to Lemma \ref{lemma_reflection}, it suffices to treat the whole plane case, since all other cases can be reduced to it by reflection.

Let $\bv\in H^1_{\rm loc}(\mbR^2)$ and $P\in H^1_{\rm loc}(\mbR^2)$ solve \eqref{Euler}. First, we derive an elliptic equation for $P$ to which the maximum principle applies.

\begin{lemma}\label{lemma_P eq}
Assume for every compact set $K$, there exists $c_K>0$ such that
\begin{equation*}
 |\bv|\ge c_K \qquad \text{a.e. in } K.
\end{equation*}
Then the pressure $P$ solves
\begin{equation}\label{eq_pressure}
 \Div\!\left(\frac{\nabla P}{|\bv|^2}\right)=0  \qquad \text{in } \mcD'(\mbR^2).
\end{equation}
\end{lemma}

\begin{proof}
Clearly,
\[
\frac{\bv}{|\bv|} \in H^1_{\rm loc}(\mbR^2).
\]
The standard $H^1$ lifting property for circle-valued maps (see, e.g., \cite{BM_book_2021}) yields an angle $\theta\in H^1_{\rm loc}(\mbR^2)$ such that
\[
\frac{\bv}{|\bv|} =(\cos\theta,\sin\theta).
\]
Direct computation yields
\[
\nabla\theta=\frac{v_1\nabla v_2-v_2\nabla v_1}{|\bv|^2}.
\]
On the other hand, it follows from the Euler system that
\begin{align*}
\nabla^\perp P=&(v_1\partial_1 v_2+v_2\partial_2 v_2, -v_1\partial_1 v_1-v_2\partial_2 v_1)\\
=&(v_1\partial_1 v_2-v_2\partial_1 v_1, v_1\partial_2 v_2-v_2\partial_2 v_1)\\
=&v_1\nabla v_2-v_2\nabla v_1.
\end{align*}
Hence, we obtain
\[
 \frac{\nabla P}{|\bv|^2} =-\nabla^\perp\theta.
\]
Then \eqref{eq_pressure} immediately follows.
\end{proof}

\begin{lemma}\label{lemma_good radii}
Let $1\le q\le2$ and let $f\in W^{1,q}_{\rm loc}(\mbR^2)$ satisfy $\nabla f\in L^q(\mbR^2).$
Then there exists a sequence $r_j\to\infty$ such that, for the Sobolev traces on $\partial B_{r_j}$,
\begin{equation*}
 \osc_{\partial B_{r_j}}f \longrightarrow 0.
\end{equation*}
\end{lemma}
\begin{proof}
Set
\[
G_q(r):=\int_{\partial B_{r}}|\nabla f|^q\,ds.
\]
Using polar coordinates gives
\[
\int_0^\infty G_q(r)\,dr=\int_\Omega|\nabla f|^q\,dx<\infty.
\]
One then readily gets
\begin{equation*}
\liminf_{r\to\infty}r^{q-1}G_q(r)=0.
\end{equation*}
Otherwise, there exist positive constants $M$ and $c$ such that $G_q(r)\ge c r^{1-q}$ for all $r\ge M$ , contradicting the divergence of $\int_M^\infty r^{1-q}\,dr$ for $q\le2$. Choose $r_j\to\infty$ realizing the liminf. Applying H\"older's inequality yields
\begin{align*}
\osc_{\partial B_{r_j}}f
&\le\int_{\partial B_{r_j}}|\nabla f|\,ds
\le (2\pi r_j)^{1-1/q}G_q(r_j)^{1/q}\\
&= (2\pi)^{1-1/q} \big(r_j^{q-1}G_q(r_j)\big)^{1/q}
\rightarrow 0.
\end{align*}
This completes the proof of the lemma.
\end{proof}

\begin{prop}\label{prop_rigidity}
Suppose $a\in L^\infty_{\rm loc}(\mbR^2)$ satisfies $a>0$ a.e. in $\mbR^2$. 
Suppose $P\in H^1_{\rm loc}(\mbR^2)$ satisfies
\begin{equation}\label{a-harmonic}
 \Div(a\nabla P)=0\qquad\text{in }\mathcal D'(\mbR^2).
\end{equation}
If
\begin{equation}\label{integrable-gredient-P}
\nabla P\in L^q(\mbR^2) \quad \text{for some } q\in[1,2],
\end{equation}
then $P$ is constant.
\end{prop}
\begin{proof}
By Lemma~\ref{lemma_good radii}, there are radii $r_j\to\infty$ such that
\[
 \osc_{\partial B_{r_j}}P\to0.
\]
Let
\[
 c_j:=\fint_{\partial B_{r_j}}P\,ds.
\]
After taking a subsequence, $c_j$ converges to some $\ell\in [-\infty,+\infty]$.

Assume that $P$ is not constant.  We show that for some sufficiently large $j$ there exists $t\in\mbR$ such that either
\begin{equation}\label{case-less}
 P\ge t\quad\text{on }\partial B_{r_j},
 \qquad |\{P<t\}\cap B_{r_j}|>0,
\end{equation}
or
\begin{equation}\label{case-greater}
 P\le t\quad\text{on }\partial B_{r_j},
 \qquad |\{P>t\}\cap B_{r_j}|>0.
\end{equation}
Here the boundary inequalities are understood for the continuous representatives of the one-dimensional Sobolev traces.

If $\ell\in\mbR$, nonconstancy implies that for some $\delta>0$ either $\{P<\ell-2\delta\}$ or $\{P>\ell+2\delta\}$ has positive measure.  Since $c_j\to\ell$ and the boundary oscillation tends to zero, choosing $t=\ell-\delta$ in the first case or $t=\ell+\delta$ in the second gives \eqref{case-less} or \eqref{case-greater}.  If $\ell=+\infty$, choose $t$ so that $\{P<t\}$ has positive measure; then \eqref{case-less} holds for all sufficiently large $j$.  The case $\ell=-\infty$ is analogous.

Suppose \eqref{case-less} holds and put
\[
 w:=(t-P)_+\in H^1(B_{r_j}).
\]
Because the trace of $w$ vanishes on $\partial B_{r_j}$, one has $w\in H^1_0(B_{r_j})$.  Since $a\in L^\infty(B_{r_j})$, $w$ is an admissible test function in \eqref{a-harmonic}.  Therefore
\[
 0=\int_{B_{r_j}}a\nabla P\cdot\nabla w\,dx
 =-\int_{\{P<t\}\cap B_{r_j}}a|\nabla P|^2\,dx.
\]
As $a>0$ a.e., it follows that $\nabla w=0$ a.e.  Since $w\in H^1_0(B_{r_j})$, Poincar\'e's inequality gives $w=0$, contradicting the second part of \eqref{case-less}.  The case \eqref{case-greater} is identical with $w=(P-t)_+$.  Thus $P$ must be constant.
\end{proof}

\begin{remark}\label{remark_two is sharp}
The Liouville theorem is sharp at the exponent $q=2$.
Indeed, define $a, P\in C^\infty(\mathbb R^2)$ by
\[
a(x)=e^{|x|^2/2}(1+|x|^2)^{3/2}, \quad P(x)=\frac{x_1}{\sqrt{1+|x|^2}}.
\]
Clearly, $a\ge1$. Direct differentiation gives
\[
\operatorname{div}(a\nabla P)=0
\qquad\text{in }\mathbb R^2.
\]
We have
\[
|\nabla P|^2=\frac{(1+x_2^2)^2+x_1^2x_2^2}
{(1+|x|^2)^3}.
\]
It follows that
\[
|\nabla P|\le\frac{1}{\sqrt{1+|x|^2}}.
\]
Consequently, $\nabla P\in L^q(\mathbb R^2)$ for every $q>2$.
On the other hand, we have
\[
|\nabla P|\ge\frac{1+x_2^2}{(1+|x|^2)^{3/2}}
\ge\frac{1}{4\sqrt{1+|x|^2}}, \qquad \text{for } |x_2|\ge\frac{|x|}{2}.
\]
It follows that $\nabla P\notin L^q(\mathbb R^2)$ for any $q\in[1,2]$.    

In view of this counterexample, we propose the following problem: does Theorem \ref{thm_new rigidity} still hold if the condition \eqref{integrable-gredient-P} is replaced by $\nabla P\in L^q(\Omega)$ for some $q\in (2,\infty)$?
\end{remark}

\begin{proof}[Proof of Theorem \ref{thm_new rigidity}]
Combining Lemma \ref{lemma_P eq} and Proposition \ref{prop_rigidity} gives the proof of Theorem \ref{thm_new rigidity}.
\end{proof}


We give a new proof of Hamel-Nadirashvili's rigidity theorem for two-dimensional steady incompressible Euler flows in the plane, the half-plane, and an infinitely long strip, under a two-sided positive bound on the speed
\begin{equation}\label{twosided}
  0<m\le |\bv(x)|\le M<\infty.
\end{equation}
Hamel-Nadirashvili's arguments are based on global streamline geometry, semilinear elliptic equations for the stream function, and, in the plane, a logarithmic growth estimate for the velocity angle.
Our proof reduces to establishing the elliptic regularity estimate $P\in\BMO(\mbR^2)$, which can then be lifted to $\nabla P\in L^2(\mbR^2)$.

For $f\in L^1_{\mathrm{loc}}(\mbR^2)$, write
\[
[f]_{\BMO}:=\sup_B\fint_B|f-f_B|\,dx,
\quad f_B:=\fint_B f\,dx:=\frac{1}{|B|}\int_B f\,dx,
\]
where the supremum is over disks.
Let us recall the following consequence of the John–Nirenberg inequality (see, e.g., \cite{Grafakos_modern_2009}),
\begin{equation}\label{JN}
 \left(\fint_B|f-f_B|^2\,dx\right)^{1/2}
 \le C_0 [f]_{\BMO}.
\end{equation}

In case of $\Omega=\mbR^2$, Corollary \ref{coro_HN rigidity} will follow from Theorem \ref{thm_new rigidity} together with the following two standard lemmas.

\begin{lemma}\label{lemma_bdd to BMO}
Let $\bF=(F_{ij})$ be a $2\times2$ matrix with entries in $L^\infty(\mbR^2)$. Suppose $P\in L^1_{\mathrm{loc}}(\mbR^2)$ solves
\begin{equation}\label{momentum-tensor}
\Div\,\bF+\nabla P=0
\quad\text{in }\mathcal D'(\mbR^2).
\end{equation}
Then $P\in\BMO(\mbR^2)$, and there exists a constant $C_1$ such that
\[
[P]_{\BMO}\le C_1\|\bF\|_{L^\infty}.
\]    
\end{lemma}
\begin{proof}
Clearly, the solution to \eqref{momentum-tensor} is unique up to an additive constant.
For existence, let $\mcR_i$ denote the $i$-th Riesz transform and define
\[
 P_0:=\sum_{i,j=1}^2 \mcR_i\mcR_j F_{ij}.
\]
Second-order Riesz transforms map $L^\infty$ to BMO (see, e.g., \cite{Grafakos_modern_2009}). Then $P_0$ is a solution to \eqref{momentum-tensor} with
\begin{equation*}
[P_0]_{\BMO}\le C_1\|\bF\|_{L^\infty}.
\end{equation*}
This proves the lemma.
\end{proof}

\begin{lemma}\label{lemma_BMO to H1}
Let $\rho$ satisfy
\[
0<\underline{\rho}\le\rho\le\bar{\rho}<\infty.
\]
Let $P\in \BMO(\mbR^2)$ solve
\begin{align}\label{eq_divrhoP}
\Div(\rho\nabla P)=0.
\end{align}
Then $\nabla P\in L^2(\mbR^2)$.
\end{lemma}
\begin{proof}
Choose $\phi\in C_c^\infty(B_{2R})$ such that $0\le\phi\le1$, $\phi=1$ on $B_R$, and $|\nabla\phi|\le 2/R$.
Testing \eqref{eq_divrhoP} with $\phi^2(P-P_{B_{2R}})$ and using Cauchy–Schwarz inequality gives
\begin{align*}
\int_{\mbR^2}\phi^2 \rho |\nabla P|^2\,dx
=&-2\int_{\mbR^2}\rho \phi (P-P_{B_{2R}}) \nabla P\cdot\nabla\phi\,dx\\
\le& 2\left( \int_{\mbR^2}\phi^2 \rho |\nabla P|^2\,dx \right)^{\frac12}
\left( \int_{\mbR^2}\rho|\nabla\phi|^2 |P-P_{B_{2R}}|^2\,dx \right)^{\frac12}.
\end{align*}
This, together with the two‑sided boundedness of $\rho$, implies that
\begin{align*}
\int_{B_R}|\nabla P|^2\,dx
\le \frac{16 \bar{\rho}}{\underline{\rho} R^2}\int_{B_{2R}}|P-P_{B_{2R}}|^2\,dx
=\frac{64\pi \bar{\rho}}{\underline{\rho}}\fint_{B_{2R}}|P-P_{B_{2R}}|^2\,dx.
\end{align*}
Then by \eqref{JN}, one concludes that
\begin{align*}
\int_{B_R}|\nabla P|^2\,dx
\le \frac{64\pi \bar{\rho} C_0^2}{\underline{\rho}} [P]_{\BMO}^2
\end{align*}
for every $R>0$. Hence $\nabla P\in L^2(\mbR^2)$.
\end{proof}

\begin{proof}[Proof of Corollary \ref{coro_HN rigidity} in $\mbR^2$]
First,
\begin{align*}
-\nabla P=\Div(\bv\otimes\bv).
\end{align*}
Since $\bv$ is bounded, it follows from Lemma \ref{lemma_bdd to BMO} that $P\in\BMO(\mbR^2)$.
Next, recall
\begin{align*}
\Div(|\bv|^{-2}\nabla P)=0.
\end{align*}
Applying Lemma \ref{lemma_BMO to H1} implies that $\nabla P\in L^2(\mbR^2)$.
Finally, Corollary \ref{coro_HN rigidity} follows from Theorem \ref{thm_new rigidity}.
\end{proof}


\section{Asymptotics of pressure}\label{sec_pressure asymptotics}

In this section, we study the far-field asymptotics of the pressure for steady Euler flows with finite total curvature. 
More precisely, we show that $P$ converges to a constant at every unbounded end of the domain.
Unlike the rigidity theorem in Section~\ref{sec_rigidity}, the result below allows stagnation points.

For any finite-Morse-index solution to the Allen-Cahn equation in $\mbR^2$, it is known that the $P$-function has a constant limit as $|x|\to\infty$.
This is because finite-Morse-index solutions are stable outside a compact set, and stable entire solutions are one-dimensional (i.e., depend on one variable only); see, e.g., \cite{GG_MA_1998,WW_CPAM_2019} for the details.
Our result is more difficult to prove for at least two reasons.
First, the Euler system \eqref{Euler} is broader than the class of autonomous semilinear elliptic equations in two dimensions.
Second, the finite-total-curvature condition is weaker than the finite-Morse-index condition. In fact, a one-dimensional periodic solution to the Allen-Cahn equation has zero total curvature, but its Morse index is infinite.

The main result of this section is stated as follows.

\begin{prop}\label{prop_pressure-asymptotics}
Suppose $\Omega\in\{\mbR^2, \mbR^2_+, \Omega_\infty\}$. Let $\bv\in H^1_{\rm loc}(\Bar{\Omega})\cap L^\infty(\Omega)$ and $P\in H^1_{\rm loc}(\Bar{\Omega})$ be a solution of \eqref{Euler} with finite total curvature. In the half-plane and strip cases, assume the slip condition \eqref{SBC} in the sense of trace. 
\begin{enumerate}[label=(\roman*)]
\item If $\Omega=\mbR^2$ or $\mbR^2_+$, then $P$ is bounded and the limit
\[
\lim_{|x|\to\infty,x\in\Bar{\Omega}}P(x)
\]
exists.

\item If $\Omega=\Omega_\infty$, then $P$ is bounded and both limits
\[
\lim_{x_1\to-\infty}P(x_1,x_2) \quad and \quad
\lim_{x_1\to+\infty}P(x_1,x_2)
\]
exist uniformly for $x_2\in[-1,1]$.
\end{enumerate}
\end{prop}

The key point of the proof is a geometric interpretation of the total curvature as the accumulated phase increment of the flow along the pressure level curves.

\begin{lemma}\label{lemma:int_levelP}
Let $\bv\in H^1(\Omega)$ and $P\in H^1(\Omega)$ be a solution of \eqref{Euler} with finite total curvature. Then
\begin{equation}\label{eq:weak-coarea-pressure}
\mcT(\bv,\Omega)
=\int_{\mbR} \left( \int_{P^{-1}(t)\cap\{|\bv|>0\}}\frac{|\nabla P|}{|\bv|^2}\,d\mcH^1\right)dt.
\end{equation}
\end{lemma}
\begin{proof}
The Euler system \eqref{Euler} implies
\begin{align*}
\nabla^\perp P=v_1\nabla v_2-v_2\nabla v_1.
\end{align*}
Hence
\[
\mcT(\bv,\Omega)=\int_{\{|\bv|>0\}}\frac{|\nabla P|^2}{|\bv|^2}\,dx.
\]
Applying the Sobolev coarea formula \cite{MSZ_TAMS_2003,PR_Selecta_2020} to the continuous $W^{2,1}_{\rm loc}$ representative of $P$ yields \eqref{eq:weak-coarea-pressure}.
\end{proof}

To estimate the inner integral on the right-hand side of \eqref{eq:weak-coarea-pressure}, we derive additional property of the velocity trace on almost every pressure level. Let $\msN$ be the null set in Lemma \ref{lemma_trace on good level}.

\begin{lemma}\label{lemma:good-level-velocity}
Let $\Omega\subset\mbR^2$ be a bounded Lipschitz domain. Suppose that $\bv\in H^1(\Omega)$ and $P\in H^1(\Omega)\cap W^{2,1}(\Omega)$ solve \eqref{Euler}.
Suppose $t\in \mbR\setminus\msN$ and $\Gamma$ is a connected component of $P^{-1}(t)$. With the orientation
\[
 \btau=\frac{\nabla^\perp P}{|\nabla P|}
\]
on $\Gamma$, it holds that
\begin{equation}\label{eq:tangential-determinant}
|\nabla P|=\bv_\Gamma\cdot\partial_s\bv_\Gamma, \qquad \text{for }\mcH^1\text{-a.e. point of }\Gamma.
\end{equation}
Furthermore, the continuous representative of $\bv_\Gamma$ satisfies
\begin{equation}\label{eq:no-stagnation-good-level}
 |\bv_\Gamma|>0\qquad\text{everywhere on }\Gamma.
\end{equation}
\end{lemma}
\begin{proof}
The Euler system \eqref{Euler} implies that
\begin{equation}\label{eq:weak-rotated-pressure}
\nabla^\perp P=v_1\nabla v_2-v_2\nabla v_1.
\end{equation}
Restricting this equation on $\Gamma$ yields
\[
|\nabla P|=(v_1\nabla v_2-v_2\nabla v_1)\cdot\btau
=\bv_\Gamma\cdot\partial_s\bv_\Gamma.
\]
for $\mcH^1$-a.e. point of $\Gamma$. This proves \eqref{eq:tangential-determinant}.

Note that $|\nabla P|$ is continuous and positive on $\Gamma$. Denote by $c_\Gamma$ its minimum on $\Gamma$, which is a positive number. Suppose that $z\in\Gamma$ is a zero of $\bv_\Gamma$. Parametrize a sub-arc of $\Gamma$ by arclength $s\in(-\ell,\ell)$ with $z=\bga(0)$, and write $\bg(s)=\bv_\Gamma (\bga(s))$. Then $\bg\in H^{1}(-\ell,\ell)$ and
\[
|\bg(s)|=|\bg(s)-\bg(0)|\le \sqrt{s} \|\bg'\|_{L^2}.
\]
This, together with \eqref{eq:tangential-determinant}, implies that
\[
0<c_\Gamma\le |\bg(s)|\,|\bg'(s)|\le \sqrt{s} |\bg'(s)| \|\bg'\|_{L^2}
\]
This contradicts $\bg'\in L^{2}(-\ell,\ell)$, so \eqref{eq:no-stagnation-good-level} follows.
\end{proof}

The next two lemmas concern the lower bounds of the coarea integral
\[
\int_{P^{-1}(t)\cap\{|\bv|>0\}}\frac{|\nabla P|}{|\bv|^2}\,d\mcH^1.
\]

\begin{lemma}\label{lemma:coarea-int-loop}
Let $t\in\mbR\setminus\msN$. If $L=L(t)$ is a closed connected component of the level set $P^{-1}(t)$, then
\begin{align*}
\int_{L}\frac{|\nabla P|}{|\bv|^2}\,ds\ge 2\pi.
\end{align*}
\end{lemma}
\begin{proof}
By Lemmas \ref{lemma_trace on good level} and \ref{lemma:good-level-velocity}, $\bv|_{L}\in H^{1}(\Gamma)$ and never vanishes. Therefore its direction admits a $H^{1}$ lifting (i.e., phase) $\theta$ along an arclength parametrization,
\[
 \frac{\bv}{|\bv|}=(\cos\theta,\sin\theta).
\]
With the orientation from Lemma \ref{lemma:good-level-velocity}, it is direct to compute that
\[
\partial_s\theta=\frac{\bv_L\cdot\partial_s\bv_L}{|\bv_L|^2} \qquad \text{a.e. on } L.
\]
Hence,
\[
\partial_s\theta
=\frac{|\nabla P|}{|\bv_L|^2}>0 \qquad \text{a.e. on } L.
\]
Since $L$ is a loop, the total increment of $\theta$ is $2\pi k$ for some integer $k\ge1$.  This proves the lemma.
\end{proof}

A connected component $W(t)$ of $P^{-1}(t)$ is said to be {\it wall-to-wall} if its two endpoints both lie on the physical boundary of the domain. An argument similar to that in Lemma \ref{lemma:coarea-int-loop} yields the following lemma.

\begin{lemma}\label{lemma:coarea-int-w2w}
Let $\Omega$ be $\Omega_\infty$ or $\mbR_+^2$. Let $t\in\mbR\setminus\msN$. If $W(t)$ is a wall-to-wall connected component of $P^{-1}(t)$, then
\begin{align*}
\int_{W(t)}\frac{|\nabla P|}{|\bv|^2}\,ds\ge \pi.
\end{align*}
\end{lemma}
\begin{proof}
By Lemma \ref{lemma:good-level-velocity}, the velocity has a continuous nonvanishing $H^1$ trace on the closed arc $W(t)$.  Its endpoint values agree with the boundary trace; by the slip condition both endpoint velocities are horizontal.  The phase $\theta$ is strictly increasing along the arc and its two endpoint values belong to $\pi\mbZ$ modulo $2\pi$.  Since the phase increment is positive, it is at least $\pi$.  The identity for $\partial_s\theta$ in the proof of Lemma \ref{lemma:coarea-int-loop} gives the desired lower bound.
\end{proof}

The next two lemmas will be needed for Proposition \ref{prop_pressure-asymptotics} in the whole plane case, and they can be easily adapted to other unbounded domains.

\begin{lemma}\label{lemma:oscP-annulus}
Suppose that $\Omega\subset\mbR^2$ is a connected bounded Lipschitz domain. Let $\bv\in H^1(\Omega)$ and $P\in H^1(\Omega)$ be a solution to \eqref{Euler}. Then it holds that
\begin{align*}
\osc_{\Omega}P\le \osc_{\partial\Omega}P+\frac{1}{2\pi}\mcT(\bv,\Omega).
\end{align*}
\end{lemma}
\begin{proof}
It is clear that
\[
\min_{\Bar{\Omega}}P \le \min_{\partial\Omega}P \le \max_{\partial\Omega}P \le \max_{\Bar{\Omega}}P.
\]
If $t<\min_{\partial\Omega}P$ or $t>\max_{\partial\Omega}P$, then $P^{-1}(t)\Subset\Omega$. 
Let $\msR(P)$ denote the set of regular values of $P$.
Then for any
\[
t\in \msR(P)\cap \left( (\min_{\Bar{\Omega}}P,\min_{\partial\Omega}P)\cup (\max_{\partial\Omega}P,\max_{\Bar{\Omega}}P) \right),
\]
$P^{-1}(t)$ consists of finitely many $C^2$ Jordan curves in $\Omega$. 
By the Morse-Sard theorem, $\msR(P)$ is a full-measure subset of the range of $P$.
Now it follows from Lemmas \ref{lemma:int_levelP} and \ref{lemma:coarea-int-loop} that
\begin{align*}
\mcT(\bv,\Omega)
=&\int_{\msR(P)}
\left(
\int_{P^{-1}(t)}\left|\nabla\left(\frac{\bv}{|\bv|} \right)\right|\,ds
\right)dt\\
\ge& 2\pi (\min_{\partial\Omega}P-\min_{\Bar{\Omega}}P+ \max_{\Bar{\Omega}}P-\max_{\partial\Omega}P)\\
=& 2\pi (\osc_{\Omega}P-\osc_{\partial\Omega}P).
\end{align*}
The desired result directly follows.
\end{proof}

\begin{lemma}\label{lemma:distP}
Suppose that $\Omega\subset\mbR^2$ is a bounded doubly connected Lipschitz domain whose boundary consists of two Jordan curves $\Gamma_1$ and $\Gamma_2$. Let $\bv\in H^1(\Omega)$ and $P\in H^1(\Omega)$ be a solution to \eqref{Euler}. Then it holds that  
\begin{align*}
\dist(P(\Gamma_1),P(\Gamma_2))\le\frac{1}{2\pi}\mcT(\bv,\Omega).
\end{align*}
\end{lemma}
\begin{proof}
If $P(\Gamma_1)\cap P(\Gamma_2)\neq \emptyset$, there is nothing to prove. Let us assume $\max_{\Gamma_1}P<\min_{\Gamma_2}P$ (the case $\max_{\Gamma_2}P<\min_{\Gamma_1}P$ can be handled similarly).
For every $t\in (\max_{\Gamma_1}P,\min_{\Gamma_2}P)$, $P^{-1}(t)\Subset\Omega$. As in the proof of Lemma \ref{lemma:oscP-annulus}, the desired inequality follows from Lemmas \ref{lemma:int_levelP} and \ref{lemma:coarea-int-loop}.
\end{proof}

\begin{lemma}\label{lemma:flat-boundary-osc}
Let $D$ be a rectangle $(a,b)\times(-1,1)$, and let $\Gamma_{\rm art}$ denote the two vertical boundaries. Suppose that $(\bv,P)$ is a weak solution in a neighborhood of $\overline D$ relative to the ambient strip and satisfies the slip condition on the physical wall. Then
\begin{equation}\label{eq:flat-osc}
 \osc_D P\le \osc_{\Gamma_{\rm art}}P+\frac1\pi\mcT(\bv,D).
\end{equation}
If $\Gamma_1$ and $\Gamma_2$ are the two vertical boundary components, then
\begin{equation}\label{eq:flat-dist}
 \dist(P(\Gamma_1),P(\Gamma_2))\le\frac1\pi\mcT(\bv,D).
\end{equation}
\end{lemma}

\begin{proof}
For \eqref{eq:flat-osc}, consider good values below $\min_{\Gamma_{\rm art}}P$ or above $\max_{\Gamma_{\rm art}}P$.  Such a level cannot meet the artificial boundary.  Every one of its components is therefore either an interior cycle or an arc whose endpoints lie on the physical wall.  By Lemmas \ref{lemma:coarea-int-loop} and \ref{lemma:coarea-int-w2w}, every nonempty such level contributes at least $\pi$ to the coarea integral.  Integrating over the two excess intervals gives \eqref{eq:flat-osc}.

For \eqref{eq:flat-dist}, suppose for example that $\max_{\Gamma_1}P<\min_{\Gamma_2}P$.  For almost every value between these two numbers, the full level set separates $\Gamma_1$ from $\Gamma_2$ and does not meet either artificial boundary.  Hence it has at least one cycle or physical-wall-to-physical-wall arc, and its coarea contribution is at least $\pi$.  Integration over the interval $(\max_{\Gamma_1}P,\min_{\Gamma_2}P)$ proves \eqref{eq:flat-dist}.  The opposite ordering is identical.
\end{proof}

We are now in a position to prove Proposition \ref{prop_pressure-asymptotics}.

\begin{proof}[Proof of Proposition \ref{prop_pressure-asymptotics}]
(\Rn{1}) Since $\bv$ is a bounded flow with finite total curvature, we see that $\nabla P\in L^2(\mbR^2)$.
By Lemma \ref{lemma_good radii}, there exists an increasing sequence of good radii $r_n\to \infty$ such that $|P(\partial B_{r_n})|$, the length of the interval $P(\partial B_{r_n})$, converges to zero.
For $m>n$, denote
\[
A_{m,n}=B_{r_m}\setminus\overline{B_{r_n}}.
\]
By Lemmas \ref{lemma:oscP-annulus} and \ref{lemma:distP}, we get, respectively,
\[
\osc_{A_{m,n}}P
\le \osc_{\partial A_{m,n}}P+\frac{1}{2\pi}\mcT(\bv,A_{m,n}),
\]
and
\[
\dist(P(\partial B_{r_n}),P(\partial B_{r_m}))\le \frac{1}{2\pi}\mcT(\bv,A_{m,n}).
\]
These, together with the elementary inequality
\[
\osc_{\partial A_{m,n}}P\le |P(\partial B_{r_n})|+|P(\partial B_{r_m})|+\dist(P(\partial B_{r_n}),P(\partial B_{r_m})),
\]
imply that
\[
\osc_{A_{m,n}}P\le |P(\partial B_{r_n})|+|P(\partial B_{r_m})|+\frac{1}{\pi}\mcT(\bv,A_{m,n}).
\]
Letting $m\to\infty$ and then $n\to\infty$ not only proves that $P$ is bounded, but also that
\[
\lim_{|x|\to\infty}P(x)
\]
exists uniformly in the exterior region.

(\Rn{2}) Let $\partial^+ B_r=\partial B_r\cap \mbR^2_+$. 
Since $\nabla P\in L^2(\mbR^2_+)$, there exists an increasing sequence $r_n\to\infty$ such that
\[
|P(\partial^+ B_{r_n})|\to 0.
\]
For $m>n$, denote
\[
A^+_{m,n}=(B_{r_m}\setminus\overline{B_{r_n}})\cap\mbR^2_+.
\]
With the aid of Lemmas \ref{lemma:int_levelP}–\ref{lemma:coarea-int-w2w} and the slip boundary conditions, adapting the proofs of Lemmas \ref{lemma:oscP-annulus} and \ref{lemma:distP} yields
\[
\osc_{A^+_{m,n}}P\le \osc_{\partial^+ B_{r_m} \cup \partial^+ B_{r_n}} P+\frac{1}{\pi}\mcT(\bv,A^+_{m,n}),
\]
and
\[
\dist(P(\partial^+ B_{r_m}),P(\partial^+ B_{r_n})) \le \frac{1}{\pi}\mcT(\bv,A^+_{m,n}).
\]
Consequently,
\[
\osc_{A^+_{m,n}}P\le |P(\partial^+ B_{r_m})|+|P(\partial^+ B_{r_n})|+ \frac{2}{\pi}\mcT(\bv,A^+_{m,n}).
\]
Letting $m\to\infty$ and then $n\to\infty$ finishes the proof.

(\Rn{3}) Define
\[
I_{x_1}:=\{x_1\}\times[-1,1], \quad x_1\in\mbR,
\]
and for $a<b$,
\[
Q_{a,b}:=(a,b)\times(-1,1).
\]
Again, adapting the proofs of Lemmas \ref{lemma:oscP-annulus} and \ref{lemma:distP} yields
\begin{equation*}
\osc_{Q_{a,b}}P\le
\osc_{I_a\cup I_b}P+\frac1\pi\mcT(\bv,Q_{a,b}),
\end{equation*}
and
\begin{equation*}
\dist(P(I_a),P(I_b))\le\frac1\pi\mcT(\bv,Q_{a,b}).
\end{equation*}
Consequently,
\begin{align}\label{oscP_strip}
\osc_{Q_{a,b}}P\le|P(I_a)|+|P(I_b)|+\frac2\pi\mcT(\bv,Q_{a,b}).    
\end{align}
Since $\nabla P\in L^2(\Omega_\infty)$, there exist sequences $a_n\to-\infty$ and $b_n\to+\infty$ such that
\begin{align}\label{boundaryP_strip}
|P(I_{a_n})|\to0, \quad |P(I_{b_n})|\to0.    
\end{align}
Combining \eqref{oscP_strip} and \eqref{boundaryP_strip} completes the proof.
\end{proof}


\section{Lower bounds via pressure oscillation}\label{sec_lower bound}

This section is dedicated to proving part (iii) of Theorem \ref{thm_FTC flow}. The proofs for the plane, half-plane, and infinitely long strip follow directly from the results established in Section \ref{sec_pressure asymptotics}. The periodic strip case, however, requires a more detailed argument.

Let 
\[
\msR(P)=(\inf P, \sup P)\setminus\msN.
\]
By Lemma \ref{lemma:int_levelP},
\begin{equation}\label{eq:T-good-levels}
\mcT(\bv,\Omega)
=\int_{\msR(P)}
\left(
\int_{P^{-1}(t)}\frac{|\nabla P|}{|\bv|^2}\,d\mcH^1
\right)dt.
\end{equation}

For $\Omega=\mbR^2$, Proposition \ref{prop_pressure-asymptotics} gives a far-field limit $P_\infty$. For every $t\in\msR(P)$ with $t\ne P_\infty$, all components of $P^{-1}(t)$ are compact and at least one is present. Lemma \ref{lemma:coarea-int-loop} therefore gives
\begin{equation}\label{plane-period}
\int_{P^{-1}(t)}\frac{|\nabla P|}{|\bv|^2}\,d\mcH^1\ge2\pi.
\end{equation}
Integrating \eqref{plane-period} in $t$ over the range of $P$ yields \eqref{ineq_sharp total curvature}.

For the half-plane and the infinite strip, Proposition \ref{prop_pressure-asymptotics} gives the corresponding far-field limit or limits. For every $t\in\msR(P)$ different from those exceptional end values, each component of $P^{-1}(t)$ is either a loop or a wall-to-wall arc. Lemmas \ref{lemma:coarea-int-loop} and \ref{lemma:coarea-int-w2w} give
\[
\int_{P^{-1}(t)}\frac{|\nabla P|}{|\bv|^2}\,d\mcH^1\ge\pi.
\]
Integrating in $t$ proves \eqref{ineq_sharp total curvature} for $\mbR^2_+$ and $\Omega_\infty$.

In the rest of this section, we consider the periodic strip case. It is enough to prove \eqref{plane-period} for every $t\in\msR(P)$.

We shall use the Bernoulli law at Sobolev regularity. Define the total head pressure
\[
H=P+\frac12|\bv|^2.
\]
Because $\bv\in H^1\cap L^\infty$, one has $H\in H^1$. We use the weak Bernoulli theorem in the precise form of \cite{KPR_ARMA_2013}: for a $H^{1}$ Euler flow, the total head pressure is constant $\mcH^1$-a.e. on every connected set on which a Sobolev stream function is constant. Thanks to Lemma \ref{lemma_reflection}, we can apply this result to a physical wall without invoking a boundary version of the theorem. The stream function is in $H^2$, and the wall segment is a connected subset of one of its level sets because of the slip boundary condition. Consequently,
\begin{equation}\label{eq:weak-Bernoulli-wall}
P+\frac12v_1^2=C_\pm
\qquad\mcH^1\text{-a.e. on }\mbT\times\{\pm1\},
\end{equation}
for suitable constants $C_\pm$. Since $P$ has a continuous trace, $C_\pm-P\ge0$ everywhere on the corresponding wall and \eqref{eq:weak-Bernoulli-wall} yields the continuous representative of the boundary speed
\begin{equation}\label{eq:boundary-speed-from-P}
|v_1|=\sqrt{2(C_\pm-P)}.
\end{equation}
By Lemma \ref{lemma:good-level-velocity}, after deleting one more null set of pressure values we may, and shall, assume that \eqref{eq:weak-Bernoulli-wall}--\eqref{eq:boundary-speed-from-P} hold at every endpoint of the good pressure level under consideration, with the endpoint velocity interpreted as the one-dimensional $H^{1}$ trace along that level.

The rest of the proof then splits into three cases.

{\bf Case 1:} $P^{-1}(t)$ contains a loop. Then \eqref{plane-period} immediately follows from Lemma \ref{lemma:coarea-int-loop}.

{\bf Case 2:} $P^{-1}(t)$ contains a wall-to-wall arc $W$, and the two endpoints of $W$ lie on different components of the physical boundary. 
Since $t$ lies strictly between the minimum and the maximum of $P$, the level set $P^{-1}(t)$ separates the nonempty sets
\[
\{P<t\}
\quad\text{and}\quad
\{P>t\}.
\]
However, a single arc $W$ terminating at different components of the boundary cannot separate the cylinder. Therefore, $P^{-1}(t)$ must contain at least one more component. 
Now, two components of $P^{-1}(t)$ contribute at least $\pi+\pi=2\pi$ to the coarea integral along $P^{-1}(t)$.
So \eqref{plane-period} still holds.

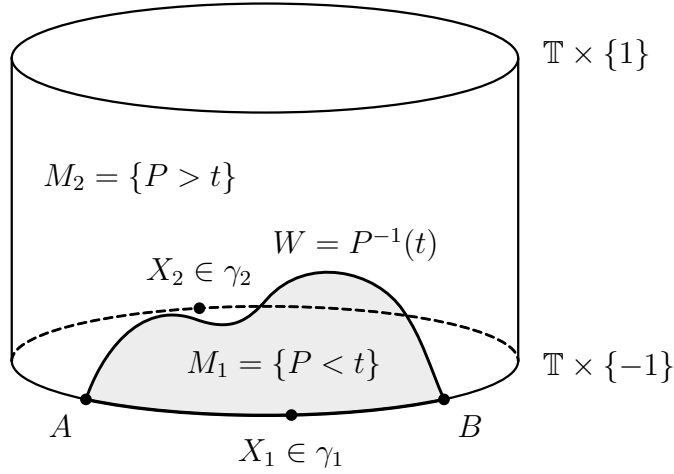
\begin{figure}[htbp]
\centering
\begin{tikzpicture}[
  x=1cm,
  y=1cm,
  line cap=round,
  line join=round,
  wall/.style={black, line width=0.85pt},
  hidden wall/.style={black, densely dashed, line width=0.75pt},
  pressure level/.style={black, line width=1.10pt},
  point/.style={circle, fill=black, inner sep=1.55pt}
]
\def\Rx{3.35}
\def\Ry{0.72}
\def\H{4}
\coordinate (A) at ({\Rx*cos(-135)},{\Ry*sin(-135)});
\coordinate (B) at ({\Rx*cos(-45)},{\Ry*sin(-45)});
\path[fill=black!7]
  (A)
  .. controls (-2.18,0.05) and (-1.68,0.82) .. (-0.92,0.55)
  .. controls (-0.45,0.38) and (-0.23,0.52) .. (0.03,0.82)
  .. controls (0.55,1.40) and (1.47,1.20) .. (1.83,0.66)
  .. controls (2.08,0.30) and (2.18,-0.08) .. (B)
  arc[start angle=-45,end angle=-135,x radius=\Rx,y radius=\Ry]
  -- cycle;
\draw[wall] (-\Rx,\H)
  arc[start angle=180,end angle=540,x radius=\Rx,y radius=\Ry];
\draw[wall] (-\Rx,0) -- (-\Rx,\H);
\draw[wall] ( \Rx,0) -- ( \Rx,\H);
\draw[hidden wall] (-\Rx,0)
  arc[start angle=180,end angle=0,x radius=\Rx,y radius=\Ry];
\draw[wall] (-\Rx,0)
  arc[start angle=180,end angle=360,x radius=\Rx,y radius=\Ry];
\draw[line width=1.25pt]
  (A) arc[start angle=-135,end angle=-45,x radius=\Rx,y radius=\Ry];
\draw[hidden wall,line width=1.05pt]
  (-\Rx,0) arc[start angle=180,end angle=0,x radius=\Rx,y radius=\Ry];
\draw[pressure level]
  (A)
  .. controls (-2.18,0.05) and (-1.68,0.82) .. (-0.92,0.55)
  .. controls (-0.45,0.38) and (-0.23,0.52) .. (0.03,0.82)
  .. controls (0.55,1.40) and (1.47,1.20) .. (1.83,0.66)
  .. controls (2.08,0.30) and (2.18,-0.08) .. (B);
\node[point] at (A) {};
\node[point] at (B) {};
\node[point] (Xone) at ({\Rx*cos(-84)},{\Ry*sin(-84)}) {};
\node[point] (Xtwo) at ({\Rx*cos(105)},{\Ry*sin(105)}) {};
\node[below left, xshift=-1pt, yshift=-2pt] at (A) {$A$};
\node[below right, xshift=1pt, yshift=-2pt] at (B) {$B$};
\node[below=5pt] at (Xone) {$X_1\in \gamma_1$};
\node[above=4pt] at (Xtwo) {$X_2\in \gamma_2$};

\node at (0.25,-0.05) {$M_1=\{P<t\}$};
\node at (-1.65,2.42) {$M_2=\{P>t\}$};
\node[above=2pt] at (1.20,1.10) {$W= P^{-1}(t)$};

\node[anchor=west] at (\Rx+0.18,\H+0.04)
  {$\mbT\times\{1\}$};
\node[anchor=west] at (\Rx+0.18,-0.08)
  {$\mbT\times\{-1\}$};

\end{tikzpicture}
\caption{The contradiction configuration in subcase 3.2. Under the assumption $P^{-1}(t)=W$, the same-wall arc $W$ separates the cylinder into $M_1=\{P<t\}$ and $M_2=\{P>t\}$. Its endpoints $A,B$ split the lower wall into $\gamma_1$ and $\gamma_2$; the opposite signs of $v_1(A)$ and $v_1(B)$ force stagnation points $X_i\in\gamma_i$, $i=1,2$.}
\label{fig:periodic-same-wall}
\end{figure}

{\bf Case 3:} $P^{-1}(t)$ contains a wall-to-wall arc $W$, and the two endpoints of $W$ lie on the same component of the physical boundary. 
Denote $A$ and $B$ the endpoints of $W$ and assume, without loss of generality, that $A,B\in\mbT\times\{-1\}$. 
Since $t$ is a regular value of $P$, the velocities are nonzero on both $A$ and $B$. The slip boundary condition implies that $\bv$ on $A$ and $B$ is horizontal. 
It remains to discuss two more subcases.

{\bf Subcase 3.1:} If $v_1(A)$ and $v_1(B)$ have the same sign, the strictly increasing phase along $W$ starts and ends at the same horizontal direction modulo $2\pi$. Then, the proof of Lemma \ref{lemma:coarea-int-loop} carries over verbatim and gives \eqref{plane-period}.

{\bf Subcase 3.2:} Assume now that $v_1(A)$ and $v_1(B)$ have different signs. 
We claim that $P^{-1}(t)$ must also contain at least one more component, so that \eqref{plane-period} holds. 
Suppose, to the contrary, that $P^{-1}(t)$ contains exactly one component, that is, $P^{-1}(t)=W$. 
Then $W$ separates $\mbT\times[-1,1]$ into two disjoint connected subsets, say $M_1=\{P<t\}$ and $M_2=\{P>t\}$. 
On the other hand, $A$ and $B$ divide the boundary component $\mbT\times\{-1\}$ into two disjoint arcs, denoted by $\gamma_1$ and $\gamma_2$. We may assume that $\gamma_i\subset M_i$, $i=1,2$.
Since $v_1(A)$ and $v_1(B)$ have different signs, on each $\gamma_i$ ($i=1,2$) there exists a stagnation point $X_i$ (See Figure \ref{fig:periodic-same-wall} for an illustration). By Bernoulli's law, there exists a constant $C$ such that
\[
P+\frac12 v_1^2=C \quad \text{on} \ \mbT\times\{-1 \}.
\]
Hence, both $X_1$ and $X_2$ are maximizers of $P$ restricted on $\mbT\times\{-1 \}$. In particular, it holds that $P(X_1)=P(X_2)>t$. This is a contradiction.

Thus, \eqref{plane-period} holds in all cases.
The proof of Theorem \ref{thm_FTC flow} is thereby concluded.


\section{Equality cases and examples}\label{sec_equality case}

In this final section, we provide examples for which \eqref{ineq_sharp total curvature} becomes equality.

In fact, the proof of the inequality also yields a criterion for equality. Recall that
\[
c_\Omega=
\begin{cases}
2\pi, &\text{if}\,\, \Omega=\mbR^2\text{ or }\Omega_{\rm per},\\
\pi, &\text{if}\,\,\Omega=\mbR^2_+\text{ or }\Omega_\infty.
\end{cases}
\]
Inequality \eqref{ineq_sharp total curvature} was obtained by inserting the pointwise-in-$t$ lower bound
\[
\int_{P^{-1}(t)}\left|\nabla\left(\frac{\bv}{|\bv|} \right)\right|\,ds \ge c_\Omega
\]
into the formula \eqref{coarea1}. 
Hence, the validity of the equality is equivalent to 
\begin{align*}
\int_{P^{-1}(t)}\left|\nabla\left(\frac{\bv}{|\bv|} \right)\right|\,ds = c_\Omega, \quad \hbox{a.e. } t.
\end{align*}
Equivalently, we have the following extremal scenarios:
\begin{enumerate}[label=(\roman*)]
\item in the plane, almost every level set of $P$ is a single loop along which the phase increment is $2\pi$;
 
\item in the half-plane or infinite strip, almost every level set of $P$ is a single wall-to-wall arc along which the phase increment is $\pi$;
 
\item in the periodic strip, the total phase increment is $2\pi$ on almost every level set of $P$.
This could be realized by a single loop of phase increment $2\pi$, by two wall-to-wall arcs of phase increment $\pi$, or by a single wall-to-wall arc of phase increment $2\pi$.
\end{enumerate}

\begin{example}[Compactly supported radial vortices]
Let
\[
    \bv(x)=V(|x|)\frac{x^\perp}{|x|},
\]
where $V$ is smooth and compactly supported with $V(0)=0$, and $x^\perp=(-x_2,x_1)$. Then $\bv$ is a steady Euler flow with pressure satisfying
\[
P'(r)=\frac{V(r)^2}{r}.
\]
Then all regular level sets of $P$ are concentric circles and the velocity direction winds once.
This example shows the sharpness of \eqref{ineq_sharp total curvature} in both the plane and the periodic strip.
\end{example}

\begin{example}[Saddle solutions to Allen--Cahn]
Let $u$ be the saddle solution to the Allen--Cahn equation $\Delta u=u^3-u$ in the plane, satisfying $x_1x_2 u>0$ whenever $x_1x_2\neq0$ (see \cite{DFP_ZAMP_1992}). Let
\[
\bv=\nabla^\perp u,\qquad
P=\frac14(1-u^2)^2-\frac12|\nabla u|^2.
\]
Since $u$ has Morse index 1 (see, e.g., \cite{WW_CPAM_2019}), $\bv$ is a steady Euler flow with finite total curvature.
Since $\bv$ has exactly one stagnation point at the origin, and a closed component of the pressure level must enclose a stagnation point (see Section \ref{sec_rigidity}), it follows that the regular pressure level curves are nested loops. 
Moreover, two distinct loops cannot share the same pressure value; otherwise, the maximum principle applied to \eqref{eq_elliptic P} would imply the existence of a stagnation point in the annular domain bounded by them.
Since the origin is the unique saddle point, the velocity winds exactly once along each regular pressure level. This example demonstrates the sharpness of \eqref{ineq_sharp total curvature} in the plane; the sharpness in the half-plane follows by restricting the solution.
\end{example}

\begin{example}[Infinite strips]
The heteroclinic constructions of De Regibus and Ruiz \cite{DR_CVPDE_2025}, and some of the least total curvature flows studied in \cite{GRXX_ARMA_2026}, provide nontrivial sharpness examples in $\Omega_\infty$. 
We refer the reader to \cite{DR_CVPDE_2025,GRXX_ARMA_2026} for the existence of a smooth solution $(\bv,P)$ to \eqref{Euler}-\eqref{SBC} with finite total curvature, for which $v_2>0$ in $\Omega_\infty$, $v_1(\cdot,1)$ is positive and increasing, and $v_1(\cdot,-1)$ is negative and decreasing.
Let $t$ be a regular value of $P$ other than the far-field limits of $P$. Since $\bv$ has no stagnation points, $P^{-1}(t)$ cannot contain any loops; it can only consist of wall-to-wall arcs. 
Using the Euler equations together with the signs and monotonicity of the boundary values of $v_1$, we have
\begin{align*}
\partial_{x_1}P(x_1,\pm 1)=-v_1\partial_{x_1}v_1(x_1,\pm 1)<0, \quad \hbox{for all } x_1\in \mbR.
\end{align*}
Thus, every regular level set $P^{-1}(t)$ consists of a single wall-to-wall arc whose endpoints lie on distinct boundary components of $\Omega_\infty$.
Finally, since $v_2>0$ in $\Omega_\infty$ and $v_1(\cdot,1)$ and $v_1(\cdot,-1)$ have different signs, the phase increment along each wall-to-wall arc is exactly $\pi$. This example demonstrates the sharpness of  \eqref{ineq_sharp total curvature} in the infinitely long strip.
\end{example}

\begin{example}[Cat's-eye flows]
Here we demonstrate the sharpness of \eqref{ineq_sharp total curvature} in the periodic strip by considering the famous cat's-eye Euler flows.
The streamline pattern for a general cat's-eye flow in $\mbT\times[-1,1]$ is given by Figure \ref{fig:cat-eye-flow}.

\begin{figure}[htbp]
\centering

\begin{tikzpicture}[x=0.63662cm,y=2.1cm]
\tikzset{
  flow/.style={
    line width=0.60pt,
    line cap=round,
    line join=round
  },
  wall/.style={
    line width=0.90pt,
    line cap=round
  },
  openflow/.style={
    flow,
    postaction={decorate},
    decoration={
      markings,
      mark=at position 0.52 with
        {\arrow{Stealth[length=2mm,width=1.5mm]}}
    }
  },
  sepflow/.style={
    flow,
    postaction={decorate},
    decoration={
      markings,
      mark=at position 0.54 with
        {\arrow{Stealth[length=2mm,width=1.5mm]}}
    }
  },
  closedflow/.style={
    flow,
    postaction={decorate},
    decoration={
      markings,
      mark=at position 0.18 with
        {\arrow{Stealth[length=2mm,width=1.5mm]}},
      mark=at position 0.68 with
        {\arrow{Stealth[length=2mm,width=1.5mm]}}
    }
  }
}
\foreach \X in {0,6.28319,12.56637}{
  \begin{scope}[shift={(\X,0)}]
    \draw[openflow]
      plot coordinates {
        (0.00000,0.79128) (0.26180,0.79257)
        (0.52360,0.79626) (0.78540,0.80187)
        (1.04720,0.80874) (1.30900,0.81618)
        (1.57080,0.82356) (1.83260,0.83038)
        (2.09440,0.83628) (2.35619,0.84104)
        (2.61799,0.84451) (2.87979,0.84661)
        (3.14159,0.84732) (3.40339,0.84661)
        (3.66519,0.84451) (3.92699,0.84104)
        (4.18879,0.83628) (4.45059,0.83038)
        (4.71239,0.82356) (4.97419,0.81618)
        (5.23599,0.80874) (5.49779,0.80187)
        (5.75959,0.79626) (6.02139,0.79257)
        (6.28319,0.79128)
      };
    \draw[openflow]
      plot coordinates {
        (0.00000,0.59548) (0.26180,0.59810)
        (0.52360,0.60556) (0.78540,0.61686)
        (1.04720,0.63062) (1.30900,0.64542)
        (1.57080,0.66001) (1.83260,0.67343)
        (2.09440,0.68501) (2.35619,0.69431)
        (2.61799,0.70108) (2.87979,0.70519)
        (3.14159,0.70656) (3.40339,0.70519)
        (3.66519,0.70108) (3.92699,0.69431)
        (4.18879,0.68501) (4.45059,0.67343)
        (4.71239,0.66001) (4.97419,0.64542)
        (5.23599,0.63062) (5.49779,0.61686)
        (5.75959,0.60556) (6.02139,0.59810)
        (6.28319,0.59548)
      };
    \draw[openflow]
      plot coordinates {
        (0.00000,0.38951) (0.26180,0.39432)
        (0.52360,0.40790) (0.78540,0.42799)
        (1.04720,0.45180) (1.30900,0.47674)
        (1.57080,0.50075) (1.83260,0.52241)
        (2.09440,0.54080) (2.35619,0.55540)
        (2.61799,0.56594) (2.87979,0.57230)
        (3.14159,0.57442) (3.40339,0.57230)
        (3.66519,0.56594) (3.92699,0.55540)
        (4.18879,0.54080) (4.45059,0.52241)
        (4.71239,0.50075) (4.97419,0.47674)
        (5.23599,0.45180) (5.49779,0.42799)
        (5.75959,0.40790) (6.02139,0.39432)
        (6.28319,0.38951)
      };

    \draw[openflow]
      plot coordinates {
        (6.28319,-0.38951) (6.02139,-0.39432)
        (5.75959,-0.40790) (5.49779,-0.42799)
        (5.23599,-0.45180) (4.97419,-0.47674)
        (4.71239,-0.50075) (4.45059,-0.52241)
        (4.18879,-0.54080) (3.92699,-0.55540)
        (3.66519,-0.56594) (3.40339,-0.57230)
        (3.14159,-0.57442) (2.87979,-0.57230)
        (2.61799,-0.56594) (2.35619,-0.55540)
        (2.09440,-0.54080) (1.83260,-0.52241)
        (1.57080,-0.50075) (1.30900,-0.47674)
        (1.04720,-0.45180) (0.78540,-0.42799)
        (0.52360,-0.40790) (0.26180,-0.39432)
        (0.00000,-0.38951)
      };

    \draw[openflow]
      plot coordinates {
        (6.28319,-0.59548) (6.02139,-0.59810)
        (5.75959,-0.60556) (5.49779,-0.61686)
        (5.23599,-0.63062) (4.97419,-0.64542)
        (4.71239,-0.66001) (4.45059,-0.67343)
        (4.18879,-0.68501) (3.92699,-0.69431)
        (3.66519,-0.70108) (3.40339,-0.70519)
        (3.14159,-0.70656) (2.87979,-0.70519)
        (2.61799,-0.70108) (2.35619,-0.69431)
        (2.09440,-0.68501) (1.83260,-0.67343)
        (1.57080,-0.66001) (1.30900,-0.64542)
        (1.04720,-0.63062) (0.78540,-0.61686)
        (0.52360,-0.60556) (0.26180,-0.59810)
        (0.00000,-0.59548)
      };

    \draw[openflow]
      plot coordinates {
        (6.28319,-0.79128) (6.02139,-0.79257)
        (5.75959,-0.79626) (5.49779,-0.80187)
        (5.23599,-0.80874) (4.97419,-0.81618)
        (4.71239,-0.82356) (4.45059,-0.83038)
        (4.18879,-0.83628) (3.92699,-0.84104)
        (3.66519,-0.84451) (3.40339,-0.84661)
        (3.14159,-0.84732) (2.87979,-0.84661)
        (2.61799,-0.84451) (2.35619,-0.84104)
        (2.09440,-0.83628) (1.83260,-0.83038)
        (1.57080,-0.82356) (1.30900,-0.81618)
        (1.04720,-0.80874) (0.78540,-0.80187)
        (0.52360,-0.79626) (0.26180,-0.79257)
        (0.00000,-0.79128)
      };

    \draw[sepflow]
      plot coordinates {
        (0.00000,0.00000) (0.19635,0.04872)
        (0.39270,0.09666) (0.58905,0.14306)
        (0.78540,0.18728) (0.98175,0.22876)
        (1.17810,0.26707) (1.37445,0.30189)
        (1.57080,0.33304) (1.76715,0.36041)
        (1.96350,0.38397) (2.15984,0.40374)
        (2.35619,0.41979) (2.55254,0.43218)
        (2.74889,0.44097) (2.94524,0.44622)
        (3.14159,0.44796) (3.33794,0.44622)
        (3.53429,0.44097) (3.73064,0.43218)
        (3.92699,0.41979) (4.12334,0.40374)
        (4.31969,0.38397) (4.51604,0.36041)
        (4.71239,0.33304) (4.90874,0.30189)
        (5.10509,0.26707) (5.30144,0.22876)
        (5.49779,0.18728) (5.69414,0.14306)
        (5.89049,0.09666) (6.08684,0.04872)
        (6.28319,0.00000)
      };

    \draw[sepflow]
      plot coordinates {
        (6.28319,0.00000) (6.08684,-0.04872)
        (5.89049,-0.09666) (5.69414,-0.14306)
        (5.49779,-0.18728) (5.30144,-0.22876)
        (5.10509,-0.26707) (4.90874,-0.30189)
        (4.71239,-0.33304) (4.51604,-0.36041)
        (4.31969,-0.38397) (4.12334,-0.40374)
        (3.92699,-0.41979) (3.73064,-0.43218)
        (3.53429,-0.44097) (3.33794,-0.44622)
        (3.14159,-0.44796) (2.94524,-0.44622)
        (2.74889,-0.44097) (2.55254,-0.43218)
        (2.35619,-0.41979) (2.15984,-0.40374)
        (1.96350,-0.38397) (1.76715,-0.36041)
        (1.57080,-0.33304) (1.37445,-0.30189)
        (1.17810,-0.26707) (0.98175,-0.22876)
        (0.78540,-0.18728) (0.58905,-0.14306)
        (0.39270,-0.09666) (0.19635,-0.04872)
        (0.00000,0.00000)
      };

    \draw[closedflow]
      plot coordinates {
        (1.08117,0.00000) (1.10654,0.05075)
        (1.18202,0.10173) (1.30575,0.15277)
        (1.47468,0.20301) (1.68466,0.25083)
        (1.93051,0.29410) (2.20618,0.33055)
        (2.50489,0.35812) (2.81927,0.37528)
        (3.14159,0.38110) (3.46391,0.37528)
        (3.77830,0.35812) (4.07700,0.33055)
        (4.35268,0.29410) (4.59853,0.25083)
        (4.80851,0.20301) (4.97744,0.15277)
        (5.10117,0.10173) (5.17664,0.05075)
        (5.20201,0.00000) (5.17664,-0.05075)
        (5.10117,-0.10173) (4.97744,-0.15277)
        (4.80851,-0.20301) (4.59853,-0.25083)
        (4.35268,-0.29410) (4.07700,-0.33055)
        (3.77830,-0.35812) (3.46391,-0.37528)
        (3.14159,-0.38110) (2.81927,-0.37528)
        (2.50489,-0.35812) (2.20618,-0.33055)
        (1.93051,-0.29410) (1.68466,-0.25083)
        (1.47468,-0.20301) (1.30575,-0.15277)
        (1.18202,-0.10173) (1.10654,-0.05075)
      } -- cycle;

    \draw[closedflow]
      plot coordinates {
        (1.70087,0.00000) (1.71861,0.04303)
        (1.77138,0.08539) (1.85790,0.12633)
        (1.97602,0.16496) (2.12285,0.20024)
        (2.29476,0.23107) (2.48752,0.25637)
        (2.69638,0.27519) (2.91621,0.28679)
        (3.14159,0.29071) (3.36697,0.28679)
        (3.58680,0.27519) (3.79567,0.25637)
        (3.98843,0.23107) (4.16034,0.20024)
        (4.30716,0.16496) (4.42529,0.12633)
        (4.51180,0.08539) (4.56458,0.04303)
        (4.58232,0.00000) (4.56458,-0.04303)
        (4.51180,-0.08539) (4.42529,-0.12633)
        (4.30716,-0.16496) (4.16034,-0.20024)
        (3.98843,-0.23107) (3.79567,-0.25637)
        (3.58680,-0.27519) (3.36697,-0.28679)
        (3.14159,-0.29071) (2.91621,-0.28679)
        (2.69638,-0.27519) (2.48752,-0.25637)
        (2.29476,-0.23107) (2.12285,-0.20024)
        (1.97602,-0.16496) (1.85790,-0.12633)
        (1.77138,-0.08539) (1.71861,-0.04303)
      } -- cycle;

    \draw[closedflow]
      plot coordinates {
        (2.25197,0.00000) (2.26293,0.02897)
        (2.29551,0.05729) (2.34894,0.08434)
        (2.42187,0.10947) (2.51254,0.13205)
        (2.61869,0.15149) (2.73771,0.16723)
        (2.86668,0.17883) (3.00243,0.18594)
        (3.14159,0.18833) (3.28076,0.18594)
        (3.41650,0.17883) (3.54547,0.16723)
        (3.66450,0.15149) (3.77065,0.13205)
        (3.86131,0.10947) (3.93425,0.08434)
        (3.98767,0.05729) (4.02026,0.02897)
        (4.03121,0.00000) (4.02026,-0.02897)
        (3.98767,-0.05729) (3.93425,-0.08434)
        (3.86131,-0.10947) (3.77065,-0.13205)
        (3.66450,-0.15149) (3.54547,-0.16723)
        (3.41650,-0.17883) (3.28076,-0.18594)
        (3.14159,-0.18833) (3.00243,-0.18594)
        (2.86668,-0.17883) (2.73771,-0.16723)
        (2.61869,-0.15149) (2.51254,-0.13205)
        (2.42187,-0.10947) (2.34894,-0.08434)
        (2.29551,-0.05729) (2.26293,-0.02897)
      } -- cycle;
    \fill (3.14159,0) circle[radius=1pt];

  \end{scope}
}
\draw[wall] (0, 1) -- (18.84956, 1);
\draw[wall] (0,-1) -- (18.84956,-1);
\end{tikzpicture}

\caption{Cat's-eye Euler flow in $\mbT\times[-1,1]$ with slip boundary conditions.}
\label{fig:cat-eye-flow}
\end{figure}
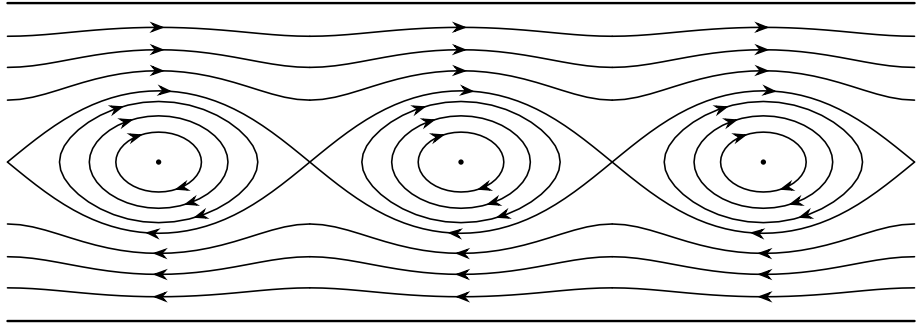

For the existence of cat's-eye flows near the Couette flow $(x_2,0)$ in $\mbT\times[-1,1]$, we refer to \cite{LZ_ARMA_2011}. We now present an explicit example of such a flow.
Let
\[
\psi=\frac{1}{\sqrt{\frac{\pi^2}{4}+1}}\cos\left(\sqrt{\frac{\pi^2}{4}+1} x_2\right)+\varepsilon\sin(x_1)\cos(\frac{\pi}{2}x_2).
\]
Then
\[
-\Delta\psi=\left( 1+\frac{\pi^2}{4} \right)\psi.
\]
So $(\bv,P)$ defined by
\begin{align*}
\bv=&\left(\sin\left(\sqrt{\frac{\pi^2}{4}+1} x_2\right)+\frac{\pi}{2}\varepsilon\sin(x_1)\sin(\frac{\pi}{2}x_2), \varepsilon\cos(x_1)\cos(\frac{\pi}{2}x_2)\right),\\
P=&-\frac12|\nabla\psi|^2-\frac12\left(1+\frac{\pi^2}{4}\right)\psi^2
\end{align*}
solve \eqref{Euler}-\eqref{SBC} in $\mbT\times[-1,1]$.
For sufficiently small $\varepsilon>0$, the vector field $\bv$ has a periodic cat's-eye structure, with exactly one cat's eye per periodic cell.
The level sets of $P$ are depicted in Figure \ref{fig:level_P}, where there are both loops and wall-to-wall arcs.

\begin{center}

\begin{figure}[htbp]
\begin{tikzpicture}[x=1.63cm,y=2.28cm,line cap=round,line join=round]
  \begin{scope}
    \draw[contour0,line width=0.58pt] (1.51048,-0.39907) -- (1.43706,-0.37602) -- (1.37412,-0.34075) -- (1.32111,-0.29730) -- (1.26922,-0.23699) -- (1.23191,-0.17375) -- (1.20592,-0.10425) -- (1.19342,-0.03475) -- (1.19418,0.04247) -- (1.20812,0.11197) -- (1.23776,0.18551) -- (1.27971,0.25112) -- (1.33216,0.30760) -- (1.39510,0.35426) -- (1.45803,0.38447) -- (1.53146,0.40239) -- (1.60489,0.40300) -- (1.67831,0.38635) -- (1.74125,0.35735) -- (1.80419,0.31225) -- (1.86196,0.25097) -- (1.90589,0.18147) -- (1.93346,0.11197) -- (1.94817,0.03475) -- (1.94740,-0.04247) -- (1.93112,-0.11969) -- (1.89859,-0.19532) -- (1.85587,-0.25869) -- (1.79370,-0.32108) -- (1.73076,-0.36321) -- (1.65733,-0.39291) -- (1.58391,-0.40457) -- (1.51048,-0.39907);
    \draw[contour1,line width=0.58pt] (1.55244,-0.77626) -- (1.51048,-0.76884) -- (1.45803,-0.74936) -- (1.36363,-0.69129) -- (1.24825,-0.58991) -- (1.14260,-0.46718) -- (1.06770,-0.35135) -- (1.01202,-0.22780) -- (0.98168,-0.11197) -- (0.97216,0.01158) -- (0.98601,0.13489) -- (1.02060,0.25097) -- (1.08041,0.37373) -- (1.16037,0.49035) -- (1.26922,0.61064) -- (1.38461,0.70635) -- (1.43706,0.73875) -- (1.47901,0.75841) -- (1.53146,0.77354) -- (1.57342,0.77703) -- (1.61538,0.77255) -- (1.66782,0.75630) -- (1.72027,0.72986) -- (1.77272,0.69517) -- (1.88810,0.59520) -- (1.99300,0.47516) -- (2.06959,0.35907) -- (2.12681,0.23552) -- (2.15993,0.11197) -- (2.16944,-0.01158) -- (2.15551,-0.13514) -- (2.11795,-0.25869) -- (2.05593,-0.38256) -- (1.97202,-0.50193) -- (1.91957,-0.56250) -- (1.85663,-0.62550) -- (1.79370,-0.67930) -- (1.74125,-0.71690) -- (1.68880,-0.74684) -- (1.64684,-0.76413) -- (1.59440,-0.77576) -- (1.55244,-0.77626);
    \draw[contour2,line width=0.58pt] (1.12528,-1.00000) -- (1.11597,-0.90734) -- (1.08349,-0.79923) -- (1.04256,-0.70656) -- (0.91587,-0.45174) -- (0.85525,-0.29730) -- (0.81903,-0.15058) -- (0.80632,-0.00386) -- (0.81776,0.14286) -- (0.85278,0.28958) -- (0.91258,0.44451) -- (1.04627,0.71429) -- (1.08645,0.80695) -- (1.11597,0.90734) -- (1.12528,1.00000);
    \draw[contour2,line width=0.58pt] (2.01631,1.00000) -- (2.02723,0.89961) -- (2.05811,0.79923) -- (2.09904,0.70656) -- (2.22377,0.45607) -- (2.28670,0.29623) -- (2.32256,0.15058) -- (2.33527,0.00386) -- (2.32382,-0.14286) -- (2.30825,-0.22008) -- (2.28670,-0.29623) -- (2.22377,-0.45607) -- (2.09904,-0.70656) -- (2.05811,-0.79923) -- (2.02723,-0.89961) -- (2.01631,-1.00000);
    \draw[contour3,line width=0.58pt] (0.89035,-1.00000) -- (0.88251,-0.89189) -- (0.85997,-0.78378) -- (0.82926,-0.68340) -- (0.73698,-0.42085) -- (0.69605,-0.27413) -- (0.67176,-0.13514) -- (0.66379,-0.00386) -- (0.67088,0.12741) -- (0.69430,0.26641) -- (0.73450,0.41313) -- (0.82926,0.68340) -- (0.85997,0.78378) -- (0.88251,0.89189) -- (0.89035,1.00000);
    \draw[contour3,line width=0.58pt] (2.25123,1.00000) -- (2.25908,0.89189) -- (2.28161,0.78378) -- (2.31232,0.68340) -- (2.40461,0.42085) -- (2.44555,0.27413) -- (2.46982,0.13514) -- (2.47780,0.00386) -- (2.47071,-0.12741) -- (2.44729,-0.26641) -- (2.40461,-0.42085) -- (2.31232,-0.68340) -- (2.28161,-0.78378) -- (2.25908,-0.89189) -- (2.25123,-1.00000);
    \draw[contour4,line width=0.58pt] (0.70743,-1.00000) -- (0.70170,-0.89189) -- (0.68359,-0.77606) -- (0.58786,-0.40541) -- (0.55692,-0.25869) -- (0.53976,-0.12741) -- (0.53412,-0.00386) -- (0.53976,0.12741) -- (0.55692,0.25869) -- (0.58786,0.40541) -- (0.68359,0.77606) -- (0.70170,0.89189) -- (0.70743,1.00000);
    \draw[contour4,line width=0.58pt] (2.43417,1.00000) -- (2.43989,0.89189) -- (2.45800,0.77606) -- (2.55373,0.40541) -- (2.58332,0.26641) -- (2.60114,0.13514) -- (2.60747,0.00386) -- (2.60184,-0.12741) -- (2.58467,-0.25869) -- (2.55373,-0.40541) -- (2.45800,-0.77606) -- (2.43989,-0.89189) -- (2.43417,-1.00000);
    \draw[contour5,line width=0.58pt] (0.54687,-1.00000) -- (0.54187,-0.88417) -- (0.52745,-0.76834) -- (0.45303,-0.39768) -- (0.43020,-0.25869) -- (0.41665,-0.12741) -- (0.41220,-0.00386) -- (0.41665,0.12741) -- (0.43020,0.25869) -- (0.45303,0.39768) -- (0.52745,0.76834) -- (0.54187,0.88417) -- (0.54687,1.00000);
    \draw[contour5,line width=0.58pt] (2.59472,1.00000) -- (2.59973,0.88417) -- (2.61414,0.76834) -- (2.68709,0.40541) -- (2.71139,0.25869) -- (2.72494,0.12741) -- (2.72940,0.00386) -- (2.72494,-0.12741) -- (2.71139,-0.25869) -- (2.68709,-0.40541) -- (2.61414,-0.76834) -- (2.59973,-0.88417) -- (2.59472,-1.00000);
    \draw[contour6,line width=0.58pt] (0.39797,-1.00000) -- (0.39417,-0.88417) -- (0.38320,-0.76834) -- (0.30878,-0.25869) -- (0.29835,-0.12741) -- (0.29491,-0.00386) -- (0.29794,0.11969) -- (0.30799,0.25097) -- (0.38320,0.76834) -- (0.39417,0.88417) -- (0.39797,1.00000);
    \draw[contour6,line width=0.58pt] (2.74362,1.00000) -- (2.74743,0.88417) -- (2.75839,0.76834) -- (2.83282,0.25869) -- (2.84325,0.12741) -- (2.84668,0.00386) -- (2.84365,-0.11969) -- (2.83361,-0.25097) -- (2.75839,-0.76834) -- (2.74743,-0.88417) -- (2.74362,-1.00000);
    \draw[contour7,line width=0.58pt] (0.25508,-1.00000) -- (0.24429,-0.76834) -- (0.19017,-0.25869) -- (0.18010,-0.00386) -- (0.18960,0.25097) -- (0.24429,0.76834) -- (0.25508,1.00000);
    \draw[contour7,line width=0.58pt] (2.88651,1.00000) -- (2.89730,0.76834) -- (2.95142,0.25869) -- (2.96149,0.00386) -- (2.95142,-0.25869) -- (2.89730,-0.76834) -- (2.88651,-1.00000);
    \draw[contour8,line width=0.58pt] (0.11440,-1.00000) -- (0.10776,-0.77606) -- (0.07244,-0.25869) -- (0.06601,-0.00386) -- (0.07244,0.25869) -- (0.10776,0.77606) -- (0.11440,1.00000);
    \draw[contour8,line width=0.58pt] (3.02719,1.00000) -- (3.03383,0.77606) -- (3.06915,0.25869) -- (3.07558,0.00386) -- (3.06915,-0.25869) -- (3.03383,-0.77606) -- (3.02719,-1.00000);
    \draw[contour9,line width=0.58pt] (6.25602,-1.00000) -- (6.23426,-0.00386) -- (6.25602,1.00000);
    \draw[contour9,line width=0.58pt] (3.16875,1.00000) -- (3.19052,0.00386) -- (3.16875,-1.00000);
    \draw[contour10,line width=0.58pt] (6.11048,-1.00000) -- (6.11686,-0.00386) -- (6.11048,1.00000);
    \draw[contour10,line width=0.58pt] (3.31429,1.00000) -- (3.30792,0.00386) -- (3.31429,-1.00000);
    \draw[contour11,line width=0.58pt] (5.95730,-1.00000) -- (5.96262,-0.76062) -- (5.99040,-0.24324) -- (5.99512,-0.00386) -- (5.99040,0.24324) -- (5.96262,0.76062) -- (5.95730,1.00000);
    \draw[contour11,line width=0.58pt] (3.46748,1.00000) -- (3.46215,0.76062) -- (3.43439,0.24324) -- (3.42965,0.00386) -- (3.43439,-0.24324) -- (3.46215,-0.76062) -- (3.46748,-1.00000);
    \draw[contour12,line width=0.58pt] (5.79122,-1.00000) -- (5.80170,-0.76834) -- (5.85699,-0.25097) -- (5.86665,-0.00386) -- (5.85699,0.25097) -- (5.80170,0.76834) -- (5.79122,1.00000);
    \draw[contour12,line width=0.58pt] (3.63356,1.00000) -- (3.62308,0.76834) -- (3.56778,0.25097) -- (3.55813,0.00386) -- (3.56778,-0.25097) -- (3.62308,-0.76834) -- (3.63356,-1.00000);
    \draw[contour13,line width=0.58pt] (5.60317,-1.00000) -- (5.60718,-0.89189) -- (5.62013,-0.77606) -- (5.68796,-0.41313) -- (5.71062,-0.26641) -- (5.72340,-0.13514) -- (5.72793,-0.00386) -- (5.72390,0.12741) -- (5.71159,0.25869) -- (5.68933,0.40541) -- (5.62013,0.77606) -- (5.60718,0.89189) -- (5.60317,1.00000);
    \draw[contour13,line width=0.58pt] (3.82161,1.00000) -- (3.81759,0.89189) -- (3.80465,0.77606) -- (3.73545,0.40541) -- (3.71327,0.25940) -- (3.70088,0.12741) -- (3.69685,0.00386) -- (3.70088,-0.12741) -- (3.71327,-0.25940) -- (3.73545,-0.40541) -- (3.80465,-0.77606) -- (3.81759,-0.89189) -- (3.82161,-1.00000);
    \draw[contour14,line width=0.58pt] (5.37260,-1.00000) -- (5.37963,-0.89189) -- (5.39999,-0.78378) -- (5.50697,-0.43086) -- (5.54385,-0.28185) -- (5.56547,-0.14286) -- (5.57309,-0.00386) -- (5.56627,0.13514) -- (5.54540,0.27413) -- (5.50762,0.42857) -- (5.39815,0.79151) -- (5.37865,0.89961) -- (5.37260,1.00000);
    \draw[contour14,line width=0.58pt] (4.05218,1.00000) -- (4.04516,0.89189) -- (4.02479,0.78378) -- (3.91717,0.42857) -- (3.88092,0.28185) -- (3.85930,0.14286) -- (3.85169,0.00386) -- (3.85850,-0.13514) -- (3.87937,-0.27413) -- (3.91717,-0.42857) -- (4.02479,-0.78378) -- (4.04516,-0.89189) -- (4.05218,-1.00000);
    \draw[contour15,line width=0.58pt] (5.01026,-1.00000) -- (5.01988,-0.92278) -- (5.05072,-0.83784) -- (5.09611,-0.75290) -- (5.24473,-0.50698) -- (5.28901,-0.42085) -- (5.32545,-0.33591) -- (5.35174,-0.25869) -- (5.37285,-0.17375) -- (5.38566,-0.08880) -- (5.39020,-0.00386) -- (5.38642,0.08108) -- (5.37579,0.15830) -- (5.35619,0.24324) -- (5.32839,0.32819) -- (5.29262,0.41313) -- (5.24960,0.49807) -- (5.09158,0.76062) -- (5.05072,0.83784) -- (5.01988,0.92278) -- (5.01026,1.00000);
    \draw[contour15,line width=0.58pt] (4.41450,1.00000) -- (4.40290,0.91506) -- (4.37041,0.83012) -- (4.32868,0.75290) -- (4.17480,0.49742) -- (4.09933,0.33591) -- (4.07077,0.25097) -- (4.05042,0.16602) -- (4.03844,0.08210) -- (4.03459,0.00386) -- (4.03834,-0.08108) -- (4.05042,-0.16602) -- (4.07077,-0.25097) -- (4.09933,-0.33591) -- (4.17480,-0.49742) -- (4.32868,-0.75290) -- (4.37041,-0.83012) -- (4.40290,-0.91506) -- (4.41450,-1.00000);
    \draw[contour16,line width=0.58pt] (4.67830,-0.47586) -- (4.76221,-0.47391) -- (4.84613,-0.45092) -- (4.91956,-0.41233) -- (4.99298,-0.35296) -- (5.05316,-0.28185) -- (5.09788,-0.20459) -- (5.12841,-0.11969) -- (5.14292,-0.02703) -- (5.13915,0.06564) -- (5.11932,0.15058) -- (5.08210,0.23552) -- (5.02970,0.31274) -- (4.96151,0.38135) -- (4.88809,0.43113) -- (4.81466,0.46206) -- (4.73075,0.47707) -- (4.64683,0.47122) -- (4.57340,0.44876) -- (4.49998,0.40884) -- (4.42655,0.34772) -- (4.36630,0.27413) -- (4.32342,0.19691) -- (4.29446,0.11197) -- (4.28185,0.02703) -- (4.28464,-0.05792) -- (4.30299,-0.14286) -- (4.33849,-0.22780) -- (4.38886,-0.30502) -- (4.44753,-0.36777) -- (4.52096,-0.42217) -- (4.59438,-0.45688) -- (4.67830,-0.47586);
  \end{scope}
  \draw[black,thick] (0,-1) rectangle (6.28319,1);
  \draw[black] (0.00000,-1) -- (0.00000,-1.035) node[below] {$0$};
  \draw[black] (3.14159,-1) -- (3.14159,-1.035) node[below] {$\pi$};
  \draw[black] (6.28319,-1) -- (6.28319,-1.035) node[below] {$2\pi$};
  \draw[black] (0,-1.00000) -- (-0.035,-1.00000) node[left] {$-1$};
  \draw[black] (0,0.00000) -- (-0.035,0.00000) node[left] {$0$};
  \draw[black] (0,1.00000) -- (-0.035,1.00000) node[left] {$1$};
  \node[anchor=west] at (6.55,1.04) {\scriptsize level $P$};
  \draw[contour0,line width=0.8pt] (6.55,0.88000) -- (6.80,0.88000) node[right,black] {\scriptsize $-0.689$};
  \draw[contour1,line width=0.8pt] (6.55,0.77000) -- (6.80,0.77000) node[right,black] {\scriptsize $-0.667$};
  \draw[contour2,line width=0.8pt] (6.55,0.66000) -- (6.80,0.66000) node[right,black] {\scriptsize $-0.646$};
  \draw[contour3,line width=0.8pt] (6.55,0.55000) -- (6.80,0.55000) node[right,black] {\scriptsize $-0.624$};
  \draw[contour4,line width=0.8pt] (6.55,0.44000) -- (6.80,0.44000) node[right,black] {\scriptsize $-0.603$};
  \draw[contour5,line width=0.8pt] (6.55,0.33000) -- (6.80,0.33000) node[right,black] {\scriptsize $-0.582$};
  \draw[contour6,line width=0.8pt] (6.55,0.22000) -- (6.80,0.22000) node[right,black] {\scriptsize $-0.560$};
  \draw[contour7,line width=0.8pt] (6.55,0.11000) -- (6.80,0.11000) node[right,black] {\scriptsize $-0.539$};
  \draw[contour8,line width=0.8pt] (6.55,0.00000) -- (6.80,0.00000) node[right,black] {\scriptsize $-0.517$};
  \draw[contour9,line width=0.8pt] (6.55,-0.11000) -- (6.80,-0.11000) node[right,black] {\scriptsize $-0.496$};
  \draw[contour10,line width=0.8pt] (6.55,-0.22000) -- (6.80,-0.22000) node[right,black] {\scriptsize $-0.475$};
  \draw[contour11,line width=0.8pt] (6.55,-0.33000) -- (6.80,-0.33000) node[right,black] {\scriptsize $-0.453$};
  \draw[contour12,line width=0.8pt] (6.55,-0.44000) -- (6.80,-0.44000) node[right,black] {\scriptsize $-0.432$};
  \draw[contour13,line width=0.8pt] (6.55,-0.55000) -- (6.80,-0.55000) node[right,black] {\scriptsize $-0.410$};
  \draw[contour14,line width=0.8pt] (6.55,-0.66000) -- (6.80,-0.66000) node[right,black] {\scriptsize $-0.389$};
  \draw[contour15,line width=0.8pt] (6.55,-0.77000) -- (6.80,-0.77000) node[right,black] {\scriptsize $-0.367$};
  \draw[contour16,line width=0.8pt] (6.55,-0.88000) -- (6.80,-0.88000) node[right,black] {\scriptsize $-0.346$};
\end{tikzpicture}
\caption{Level curves of $P$, $\varepsilon=0.1$.}
\label{fig:level_P}
\end{figure}
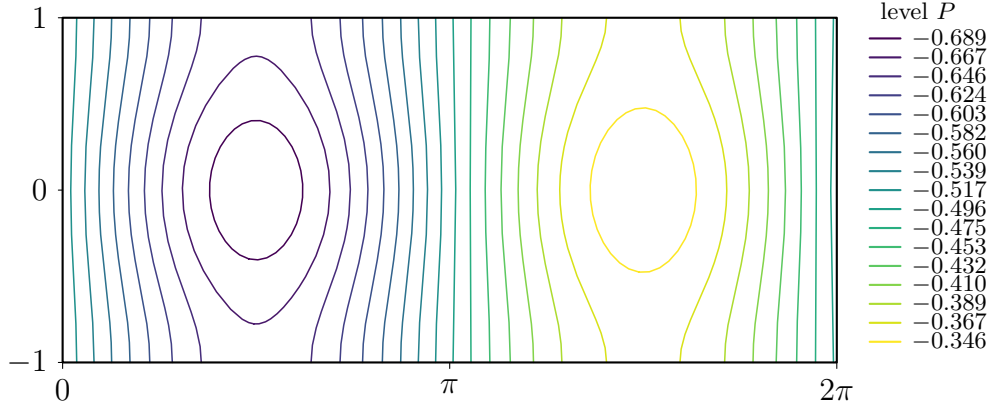
\end{center}

We consider cat's-eye flows with exactly one cat's eye in each period.
Compared to the previous examples, the level curves of the cat's-eye flow are more complex, which makes it more difficult to compute the total curvature using the coarea formula \eqref{coarea1}.
Nevertheless, the streamline pattern shown in Figure \ref{fig:cat-eye-flow} allows us to apply the coarea formula \eqref{coarea2} to compute the total curvature.

More precisely, we recall the equidistribution formula for the total curvature established in \cite{GXX_CMP_2026}:
\begin{equation}\label{equi-distribution}
\mcT(\bv,\Omega_{\rm per})=
2\pi \int_{\Sigma_y}|\nabla P|\,d\mcH^1, \quad \hbox{a.e. } y\in \mbS^1,
\end{equation}
where
\[
\Sigma_y:=\left\{x\in\Omega_{\rm per}: |\bv(x)|>0, \frac{\bv(x)}{|\bv(x)|}=y\right\}.
\]

Let $z_s$ and $z_e$ denote the saddle and elliptic stagnation points of the cat's-eye flow, respectively.
Here, both stagnation points are assumed to be nondegenerate.
Their local structures then ensure the existence of simply-connected neighborhoods $\mcN(z_s)$ and $\mcN(z_e)$ of $z_s$ and $z_e$, respectively, such that the following hold: (\Rn{1}) $\frac{\bv}{|\bv|}:\partial\mcN(z_s)\to\mbS^1$ and $\frac{\bv}{|\bv|}:\partial\mcN(z_e)\to\mbS^1$ are bijections; (\Rn{2}) for every direction $y\in\mbS^1$, $\Sigma_y\cap \mcN(z_s)$ is an arc ending at $z_s$ and at a point on $\partial\mcN(z_s)$, and $\Sigma_y\cap \mcN(z_e)$ is an arc ending at $z_e$ and at a point on $\partial\mcN(z_e)$.

Next, the Morse–Sard theorem implies that almost every $y\in S^1$ is a regular value of the direction map
\[
\frac{\bv}{|\bv|}: \Omega_{\rm per}\setminus(\mcN(z_s)\cup \mcN(z_e))\to \mbS^1.
\]
Let $y$ be a regular value of the above mapping and take any connected component $\sigma_y$ of $\Sigma_y$. Note that $\sigma_y$ is also a level curve of the phase function $\theta$ of $\bv$. The analysis in the proof of Lemma \ref{lemma:coarea-int-loop} then shows that, with $\sigma_y$ suitably oriented, the pressure $P$ is strictly increasing along $\sigma_y$, and that the tangent derivative of $P$ along $\sigma_y$ is precisely $|\nabla P|$. In particular, the monotonicity of $P$ implies that $\Sigma_y$ cannot contain any loops.

The above analysis, together with the slip boundary condition and the absence of stagnation points on the boundary, implies that for a.e. $y\in \mbS^1\setminus\{(1,0),(-1,0)\}$, $\Sigma_y\Subset\Omega_{\rm per}$ is a single arc connecting $z_s$ and $z_e$. Moreover, $\partial_s P=|\nabla P|>0$ on $\Sigma_y$, where the tangent derivative $\partial_s$ is defined with the unit tangent vector pointing from $z_e$ to $z_s$ as arc length increases.
In fact, from the Euler equations and the fact that $|\nabla P(z_s)|=|\nabla P(z_e)|=0$, we obtain 
\[
\nabla^2 P(z_s)=-(\nabla \bv(z_s))^2
=\det(\nabla\bv(z_s))I_{2}<0, \quad \nabla^2 P(z_e)>0,
\]
where $\nabla\bv$ denotes the matrix $(\partial_{x_i}v_j)_{ij}$. Consequently, $P$ attains a local maximum at $z_s$ and a local minimum at $z_e$.

We now get from \eqref{equi-distribution} that
\begin{align*}
\mcT(\bv,\Omega_{\rm per})=
2\pi \int_{\Sigma_y}\partial_s P \,d s
=2\pi (P(z_s)-P(z_e)).
\end{align*}
On the other hand, \eqref{ineq_sharp total curvature} tells that
\begin{align*}
\mcT(\bv,\Omega_{\rm per})\ge 
2\pi (\max P-\min P).
\end{align*}
Together, these imply that $P(z_s)=\max P$, $P(z_e)=\min P$, and hence
\begin{align*}
\mcT(\bv,\Omega_{\rm per})= 2\pi (\max P-\min P).
\end{align*}
\end{example}


\addtocontents{toc}{\protect\setcounter{tocdepth}{0}}
\section*{Acknowledgments}
\addtocontents{toc}{\protect\setcounter{tocdepth}{2}}
The research of  Gui is supported by NSFC Key Program (Grant No. 12531010), University of Macau research grants CPG2024-00016-FST, CPG2025-00032-FST, SRG2023-00011-FST, MYRG-GRG2023-00139-FST-UMDF, UMDF Professorial Fellowship of Mathematics, Macao SAR FDCT 0003/2023/RIA1 and Macao SAR FDCT 0024/2023/RIB1.
The research of  Li is supported by the National Natural Science Fund of China for Excellent Young Scholars (No. 12522109) and the National Science Fund of China General Program (No. 12471105).
The research of Xie is partially supported by NSFC grants 12571238 and 12426203.


\bibliographystyle{plain}

\end{document}